\documentclass[12pt,reqno]{amsart}

\usepackage{mathtools}

\usepackage{mathdots}

\usepackage{color}

\usepackage{pdfsync}

\usepackage{enumitem}

\usepackage{wasysym}

\usepackage{amssymb}

\usepackage{mathrsfs}

\DeclareMathAlphabet{\mathpzc}{OT1}{pzc}{m}{it}

\usepackage[all]{xy}

\usepackage{tikz}
\usetikzlibrary{arrows,matrix,decorations.pathmorphing,decorations.pathreplacing,positioning,shapes.geometric,shapes.misc,decorations.markings,decorations.fractals,calc,patterns}

\usepackage{graphicx}

\usepackage{float}

\usepackage[bottom]{footmisc}

\usepackage{moreenum}

\usepackage{makecell}

\entrymodifiers={+!!<0pt,\fontdimen22\textfont2>}

\def\cC{\mathscr{C}}
\def\cD{\mathscr{D}}

\def\cF{\mathscr{F}}

\def\cP{\mathscr{P}}

\def\cS{\mathscr{S}}
\def\cT{\mathscr{T}}

\def\cX{\mathscr{X}}
\def\cY{\mathscr{Y}}
\def\cZ{\mathscr{Z}}

\def\BZ{\mathbb{Z}}

\def\add{\operatorname{add}}

\def\adots{\mathinner{\mkern1mu\raise1.0pt\vbox{\kern7.0pt\hbox{.}}\mkern2mu\raise4.0pt\hbox{.}\mkern2mu\raise7.0pt\hbox{.}\mkern1mu}}

\def\ast{{\textstyle *}}

\def\dddots{\mathinner{\mkern1mu\raise10.0pt\vbox{\kern7.0pt\hbox{.}}\mkern2mu\raise5.3pt\hbox{.}\mkern2mu\raise1.0pt\hbox{.}\mkern1mu}}
\def\dddotssmall{\mathinner{\mkern1mu\raise7.0pt\vbox{\kern7.0pt\hbox{.}}\mkern-1mu\raise4pt\hbox{.}\mkern-1mu\raise1.0pt\hbox{.}\mkern1mu}}

\def\Db{\cD^{\operatorname{b}}}

\def\dim{\operatorname{dim}}

\def\dual{\operatorname{D}}
\def\End{\operatorname{End}}

\def\Ext{\operatorname{Ext}}

\def\Hom{\operatorname{Hom}}
\def\id{\operatorname{id}}
\def\Image{\operatorname{Im}}

\def\indec{\operatorname{ind}}
\def\index{\operatorname{index}}
\def\indexhigher{\operatorname{index}^{\,\tiny{\pentagon}}}

\def\K{\operatorname{K}}
\def\K0{\operatorname{K}_0}

\def\Ksp{\operatorname{K}_0^{\operatorname{sp}}}

\def\mod{\operatorname{mod}}

\def\obj{{\operatorname{obj}}}

\def\PSL2{\operatorname{PSL}_2}

\def\rad{\operatorname{rad}}

\def\SL2{\operatorname{SL}_2}

\numberwithin{equation}{section}

\newtheorem{Lemma}{Lemma}[section]
\newtheorem{Theorem}[Lemma]{Theorem}
\newtheorem{Proposition}[Lemma]{Proposition}

\theoremstyle{definition}
\newtheorem{Definition}[Lemma]{Definition}
\newtheorem{Setup}[Lemma]{Setup}

\newtheorem{Remark}[Lemma]{Remark}

\newtheorem*{bfhpg*}{}

  {\begin{list}{}{%
    \settowidth{\labelwidth}{\textbf{#1:}}%
    \setlength{\leftmargin}{\labelwidth}\addtolength{\leftmargin}{\labelsep}}}%
  {\end{list}}

\begin{document}

\setlength{\parindent}{0pt}
\setlength{\parskip}{7pt}

\title[Higher tropical friezes and indices]{Tropical friezes and the index in higher homological algebra}

\author{Peter J\o rgensen}

\address{School of Mathematics, Statistics and Physics, Newcastle University, Newcastle upon Tyne NE1 7RU, United Kingdom}
\email{peter.jorgensen@ncl.ac.uk}

\urladdr{http://www.staff.ncl.ac.uk/peter.jorgensen}





\begin{abstract} 

Cluster categories and cluster algebras encode two dimensional structures.  For instance, the Auslander--Reiten quiver of a cluster category can be drawn on a surface, and there is a class of cluster algebras determined by surfaces with marked points.  

\medskip
\noindent
Cluster characters are maps from cluster categories (and more general triangulated categories) to cluster algebras.  They have a tropical shadow in the form of so-called tropical friezes, which are maps from cluster categories (and more general triangulated categories) to the integers.

\medskip
\noindent
This paper will define higher dimensional tropical friezes.  One of the motivations is the higher dimensional cluster categories of Oppermann and Thomas, which encode $( d+1 )$-dimensional structures for an integer $d \geqslant 1$.  They are $( d+2 )$-angulated categories, which belong to the subject of higher homological algebra.

\medskip
\noindent
We will define higher dimensional tropical friezes as maps from higher cluster categories (and more general $( d+2 )$-angulated categories) to the integers.  Following Palu, we will define a notion of $( d+2 )$-angulated index, establish some of its properties, and use it to construct higher dimensional tropical friezes.

\end{abstract}

\maketitle

\setcounter{section}{-1}
\section{Introduction}
\label{sec:introduction}

The background of this paper is a certain part of cluster theory, summed up in Figure \ref{fig:clusters}.  Let us give some brief reminders about the objects in the figure.
\begin{itemize}
\setlength\itemsep{4pt}

  \item  Cluster categories are a class of triangulated categories defined by Buan, Marsh, Reineke, Reiten, and Todorov in \cite[sec.\ 1]{BMRRT}.  

  \item  Cluster algebras are a class of commutative algebras defined by Fomin and Zelevinsky in \cite[def.\ 2.3]{FZ}.

  \item  Cluster characters are certain ``nice'' maps from cluster categories (and more general triangulated categories) to cluster algebras.  They were defined by Palu in \cite[def.\ 1.2]{Palu}.  See Section \ref{subsec:classic} below.
  
  \item  Tropical friezes are certain ``nice'' maps from cluster categories (and more general triangulated categories) to the integers.  They can be viewed as a tropical version of cluster characters, and were defined by Guo in \cite[def.\ 2.2]{G},  following the lead of Propp \cite[sec.\ 5]{Propp} and Coxeter \cite[sec.\ 1]{C}.  See Section \ref{subsec:classic} below.
  
\end{itemize}
All these objects are, in a sense, two dimensional.  For instance, Fomin, Shapiro, and Thurston discovered a class of cluster algebras arising from surfaces with marked points, see \cite[sec.\ 5]{FST}.  Moreover, the Auslander--Reiten (=AR) quivers of cluster categories are two dimensional because they consist locally of two dimensional meshes.  An example is shown at the left of Figure \ref{fig:meshes}; it reflects that the cluster category is triangulated, and that there is an AR triangle $s_2 \rightarrow x_1 \oplus x'_1 \rightarrow s_0 \rightarrow \Sigma s_2$, where $\Sigma$ denotes the suspension functor.

It is not yet known how to obtain higher dimensional generalisations of cluster algebras and cluster characters, but this paper will give a higher dimensional generalisation of the part of Figure \ref{fig:clusters} dealing with tropical friezes.  This is based on work by Oppermann and Thomas, who defined higher dimensional cluster categories, see \cite[sec.\ 5]{OT} and Section \ref{subsec:classic2} below.  They are $( d+2 )$-angulated categories as defined by Geiss, Keller, and Oppermann for an integer $d \geqslant 1$, see \cite[def.\ 2.1]{GKO}, and hence belong to the subject of higher homological algebra introduced by Iyama.  They are, in a sense, $( d+1 )$-dimensional objects.  For instance, their AR quivers are $( d+1 )$-dimensional because they consist locally of $( d+1 )$-dimensional meshes.  An example is shown for $d = 2$ at the right of Figure \ref{fig:meshes}; it reflects that the category is $4$-angulated, and that there is an AR $4$-angle $s_3 \rightarrow x_2 \oplus x'_2 \oplus x''_2 \rightarrow x_1 \oplus x'_1 \oplus x''_1 \rightarrow s_0 \rightarrow \Sigma^2 s_3$.

We will define higher dimensional tropical friezes as maps from higher cluster categories (and more general $( d+2 )$-angulated categories) to the integers.  The definition is inspired by Oppermann and Thomas's higher dimensional tropical exchange relations in Dynkin type $A$, see \cite[thm.\ 1.4]{OT}.  We will define a notion of $( d+2 )$-angulated index, following Palu's definition of the index on triangulated categories, see \cite[sec.\ 2.1]{Palu}.  We will establish some properties of the $( d+2 )$-angulated index, and show that it can be used to construct higher dimensional tropical friezes; this generalises a result by Guo, see \cite[thm.\ 3.1]{G}.
\begin{figure}
\begin{center}
\begin{tikzpicture}[scale=1]
  \draw (-5,0) ellipse (2cm and 1cm);
  \draw (5,1.5) ellipse (2cm and 1cm);
  \node at (-5,0){Cluster categories};
  \node at (5,1.5){Cluster algebras};
  \node at (5,-1.5){$\BZ$};
  \draw[->] (-2.5,0.5) -- (2.5,1.5) node[midway,sloped,above] {Cluster characters};
  \draw[->] (-2.5,-0.5) -- (4.5,-1.5) node[midway,sloped,below] {Tropical friezes};
\end{tikzpicture}
\end{center}
\caption{The background to this paper is a certain part of cluster theory, summed up in this figure.}
\label{fig:clusters}
\end{figure}
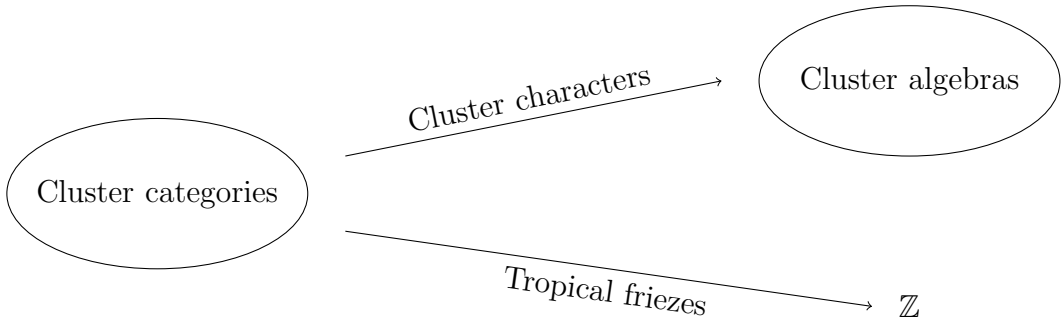

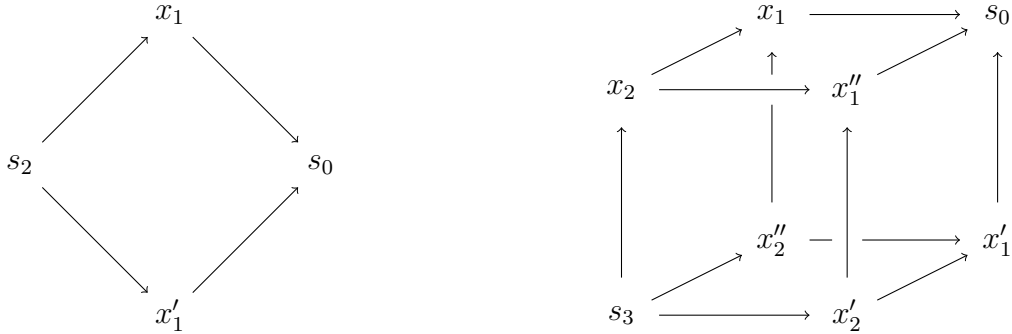
\begin{figure}
\begin{center}
\begin{tikzpicture}[scale=1]
  \draw (-8,2) node {$s_2$};
  \draw (-6,0) node {$x'_1$};
  \draw (-6,4) node {$x_1$};
  \draw (-4,2) node {$s_0$};
  \draw[->] (-7.7,2.3) -- (-6.3,3.7);
  \draw[->] (-7.7,1.7) -- (-6.3,0.3);
  \draw[->] (-5.7,0.3) -- (-4.3,1.7);
  \draw[->] (-5.7,3.7) -- (-4.3,2.3);
  
  \draw (0,0) node {$s_3$};
  \draw (3,0) node {$x'_2$};
  \draw (0,3) node {$x_2$};
  \draw (2,1) node {$x''_2$};
  \draw (3,3) node {$x''_1$};
  \draw (5,1) node {$x'_1$};  
  \draw (2,4) node {$x_1$};  
  \draw (5,4) node {$s_0$};
  \draw[->] (0.5,0) -- (2.5,0);
  \draw[->] (0,0.5) -- (0,2.5);
  \draw[->] (0.4,0.2) -- (1.6,0.8);
  \draw[->] (0.5,3) -- (2.5,3);
  \draw[->] (3,0.5) -- (3,2.5);
  \draw[->] (0.4,3.2) -- (1.6,3.8);  
  \draw[->] (3.4,0.2) -- (4.6,0.8);
  \draw[->] (3.4,3.2) -- (4.6,3.8);  
  \draw[->] (2.5,4) -- (4.5,4);
  \draw[->] (5,1.5) -- (5,3.5);
  \draw (2.5,1) -- (2.8,1); \draw[->] (3.2,1) -- (4.5,1);  
  \draw (2,1.5) -- (2,2.8); \draw[->] (2,3.2) -- (2,3.5);  
\end{tikzpicture}
\end{center}
\caption{The two dimensional mesh is a typical local shape in the Auslander--Reiten (= AR) quiver of a triangulated category, reflecting that there is an AR triangle $s_2 \rightarrow x_1 \oplus x'_1 \rightarrow s_0 \rightarrow \Sigma s_2$.  The three dimensional mesh is a typical local shape in the AR quiver of a $4$-angulated category, reflecting that there is an AR $4$-angle $s_3 \rightarrow x_2 \oplus x'_2 \oplus x''_2 \rightarrow x_1 \oplus x'_1 \oplus x''_1 \rightarrow s_0 \rightarrow \Sigma^2 s_3$.}
\label{fig:meshes}
\end{figure}

\subsection{Background: Tropical friezes on triangulated categories}
\label{subsec:classic}

Let $k$ be a field, $\cC$ a $k$-linear $\Hom$-finite triangulated category with split idempotents.  Write $\cC( -,- )$ for $\Hom_{ \cC }( -,- )$.  Assume that $\cC$ is $2$-Calabi--Yau, in the sense that there are natural isomorphisms $\cC( s,s' ) \cong \dual\!\cC\big( s',\Sigma^2 s \big)$ for $s,s' \in \cC$, where $\dual( - ) = \Hom_k( -,k )$ is $k$-linear duality.  Let $\indec\,\cC$ denote the indecomposable objects of $\cC$.  The objects $s_0$, $s_2 \in \indec\, \cC$ are called an {\em exchange pair in} $\cC$ if $\dim_k \cC( s_0,\Sigma s_2 ) = 1$.  Then we also have $\dim_k \cC( s_2,\Sigma s_0 ) = 1$ since $\cC$ is $2$-Calabi--Yau, so there are non-split triangles
\begin{equation}
\label{equ:classic_exchange_triangles}
  s_2 \rightarrow x_1 \rightarrow s_0 \rightarrow \Sigma s_2
  \;\;,\;\;
  s_0 \rightarrow y_1 \rightarrow s_2 \rightarrow \Sigma s_0
\end{equation}
in $\cC$.

A map $\chi : \obj\,\cC \rightarrow A$ to a commutative ring $A$, which is constant on isomorphism classes and exponential in the sense $\chi( s \oplus s' ) = \chi( s )\chi( s' )$, is called a {\em cluster character on $\cC$} if it satisfies
\[
  \chi( s_2 )\chi( s_0 ) = \chi( x_1 ) + \chi( y_1 )
\]
for each exchange pair $s_0, s_2 \in \indec\, \cC$ with ensuing non-split triangles \eqref{equ:classic_exchange_triangles}, see \cite[def.\ 1.2]{Palu}.  In typical applications, $\cC$ is a cluster category, $A$ a ring of Laurent polynomials, and $\chi$ maps the rigid indecomposable objects of $\cC$ to the cluster variables of a cluster algebra contained in $A$.

Tropical friezes are tropicalised versions of cluster characters.  A map $f : \obj\,\cC \rightarrow \BZ$, which is constant on isomorphism classes and additive in the sense $f( s \oplus s' ) = f( s ) + f( s' )$, is called a {\em tropical frieze on} $\cC$ if it satisfies
\begin{equation}
\label{equ:classic_tropical_exchange_relation}
  f( s_0 ) + f( s_2 ) = \max \{\, f( x_1 ),f( y_1 ) \,\}
\end{equation}
for each exchange pair $s_0, s_2 \in \indec\, \cC$ with ensuing non-split triangles \eqref{equ:classic_exchange_triangles}, see \cite[def.\ 2.2]{G}.

\subsection{Background: Higher dimensional cluster categories}
\label{subsec:classic2}

Let $\Phi$ be a finite dimensional $k$-algebra, which is $d$-representation finite in the sense of \cite[def.\ 2.2]{IO1}.  Then $\mod( \Phi )$, the category of finitely generated right modules, has a unique $d$-cluster tilting subcategory $\cF$, see \cite[prop.\ 2.3]{IO1}.  

Oppermann and Thomas defined the {\em higher dimensional cluster category} $\cS$ of $\Phi$, see \cite[sec.\ 5]{OT}, which is a $( d+2 )$-angulated category, see \cite[def.\ 2.1]{GKO}.  There is an inclusion $\cF \subseteq \cS$.  If $d = 1$, then $\Phi$ is a hereditary algebra of finite representation type, $\cF$ is all of $\mod( \Phi )$, and $\cS$ is the usual triangulated cluster category from \cite[sec.\ 1]{BMRRT}.  If $d \geqslant 2$, then $\cS$ is a new object where the suspension functor $\Sigma$ has been replaced by a $d$-suspension functor $\Sigma^d$ (not necessarily a $d$th power), and mapping cones have been replaced by complexes consisting of $d$ objects.

\subsection{Tropical friezes on $( d+2 )$-angulated categories}
\label{subsec:frieze}

Let $d \geqslant 1$ be an integer, $\cS$ a $k$-linear $\Hom$-finite $( d+2 )$-angulated category with split idempotents and $d$-suspension functor $\Sigma^d$.  Assume that $\cS$ is $2d$-Calabi--Yau, in the sense that there are natural isomorphisms $\cS( s,s' ) \cong \dual\!\cS\big( s',( \Sigma^d )^2 s \big)$ for $s,s' \in \cS$; this is satisfied by higher dimensional cluster categories by \cite[thm.\ 5.2(2)]{OT}.

The objects $s_0$, $s_{ d+1 } \in \indec\, \cS$ are called an {\em exchange pair in $\cS$} if $\dim_k \cS( s_0,\Sigma^d s_{ d+1 } ) = 1$.  Then we also have $\dim_k \cS( s_{ d+1 },\Sigma^d s_0 ) = 1$ since $\cS$ is $2d$-Calabi--Yau, so there are $( d+2 )$-angles in $\cS$:
\begin{align}
\label{equ:exchange_pair1}
  & s_{ d+1 } \rightarrow x_d \rightarrow \cdots \rightarrow x_1 \rightarrow s_0 \xrightarrow{ \gamma_0 } \Sigma^d s_{ d+1 }, \\
\label{equ:exchange_pair2}
  & s_0 \rightarrow y_1 \rightarrow \cdots \rightarrow y_d \rightarrow s_{ d+1 } \xrightarrow{ \gamma_{ d+1 } } \Sigma^d s_0
\end{align}
where $\gamma_0, \gamma_{ d+1 } \neq 0$.
\begingroup
\renewcommand{\theLemma}{\Alph{Lemma}}
\begin{Definition}
\label{def:frieze}
A map $f : \obj\,\cS \rightarrow \BZ$, which is constant on isomorphism classes and additive in the sense $f( s \oplus s' ) = f( s ) + f( s' )$, is a {\em tropical frieze on $\cS$} if it satisfies
\begin{equation}
\label{equ:frieze_defining_equation}
  f( s_0 ) + (-1)^{ d+1 }f( s_{ d+1 } )
  = \max \Bigg\{ \sum_{ i=1 }^d (-1)^{ i+1 }f( x_i ),\sum_{ i=1 }^d (-1)^{ i+1 }f( y_i ) \Bigg\}
\end{equation}
for each exchange pair $s_0, s_{ d+1 } \in \indec\, \cS$ with $( d+2 )$-angles \eqref{equ:exchange_pair1} and \eqref{equ:exchange_pair2} where $\gamma_0, \gamma_{ d+1 } \neq 0$.
\hfill $\Box$
\end{Definition}
\endgroup
If $d = 1$ then Definition \ref{def:frieze} specialises to the definition of tropical friezes on triangulated categories:  The $( d+2 )$-angles \eqref{equ:exchange_pair1} and \eqref{equ:exchange_pair2} become the triangles \eqref{equ:classic_exchange_triangles}, and Equation \eqref{equ:frieze_defining_equation} becomes Equation \eqref{equ:classic_tropical_exchange_relation}.  For general $d$, Definition \ref{def:frieze} is inspired by Oppermann and Thomas's higher dimensional tropical exchange relations in Dynkin type $A$, see \cite[thm.\ 1.4]{OT}.  However, our definition does not specialise to theirs because of the boundary terms included in \cite[thm.\ 1.4]{OT}.

A concrete example of Definition \ref{def:frieze} is given in Figure \ref{fig:frieze2}, which shows the values of a tropical frieze $f$ on the AR quiver of a certain $5$-angulated category.  Note that the quiver is, in fact, a four dimensional object, but has sufficiently few vertices to be drawn as shown.
\begin{figure}
\begin{tikzpicture}[scale=3]
  \node at (0:1.0){$-17$};
  \draw[->] (6:1.0) arc (6:33:1.0);
  \node at (40:1.0){$-8$};
  \draw[->] (46:1.0) arc (46:74:1.0);
  \node at (80:1.0){$2$};
  \draw[->] (86:1.0) arc (86:114:1.0);
  \node at (120:1.0){$19$};
  \draw[->] (126:1.0) arc (126:154:1.0);
  \node at (160:1.0){$26$};
  \draw[->] (166:1.0) arc (166:194:1.0);
  \node at (200:1.0){$17$};
  \draw[->] (206:1.0) arc (206:234:1.0);
  \node at (240:1.0){$8$};
  \draw[->] (246:1.0) arc (246:273:1.0);
  \node at (280:1.0){$-2$};
  \draw[->] (287:1.0) arc (287:312:1.0);
  \node at (320:1.0){$-19$};
  \draw[->] (328:1.0) arc (328:354:1.0);
\end{tikzpicture}
\caption{A tropical frieze $f$ on the Auslander--Reiten quiver of a certain $5$-angulated category.  If $a,b,c,d,e$ are consecutive objects, then $f( a ) + f( e ) = \max \{\, f( b ) - f( c ) + f( d ),0 \,\}$.  See Section \ref{sec:example} for further details.}
\label{fig:frieze2}
\end{figure}
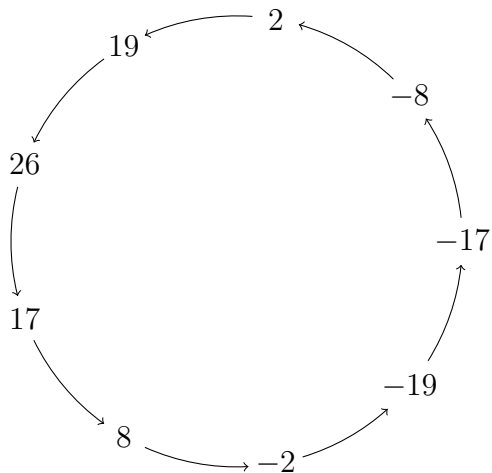
In this case, Equation \eqref{equ:frieze_defining_equation} means that if $a,b,c,d,e$ are consecutive objects in the AR quiver, then
\[
  f( a ) + f( e )
  = \max \{\, f( b ) - f( c ) + f( d ),0 \,\}.
\]
See Section \ref{sec:example} for further details.

Observe that a different notion of higher dimensional friezes was defined by McMahon, see \cite[sec.\ 4]{McMahon}.

\subsection{The $( d+2 )$-angulated index}
\label{subsec:index}

Let us drop the assumption that $\cS$ is $2d$-Calabi--Yau, which is not needed for defining the $( d+2 )$-angulated index.  Let $t \in \cS$ be an Oppermann--Thomas cluster tilting object, see \cite[def.\ 5.3]{OT} or Definition \ref{def:OT_cluster_tilting}, and set $\cT = \add( t )$. 
\begingroup
\renewcommand{\theLemma}{\Alph{Lemma}}
\begin{Definition}
\label{def:index2}
For $s \in \cS$, let
\begin{equation}
\label{equ:index2_higher_angle}
  t_d \xrightarrow{ \tau_d } \cdots \xrightarrow{ \tau_1 } t_0 \rightarrow s \rightarrow \Sigma^d t_d
\end{equation}
be a $( d+2 )$-angle in $\cS$ with $t_i \in \cT$ and $\tau_i$ in the radical of $\cS$ for each $i$.  The {\em $( d+2 )$-angulated index of $s$ with respect to $t$} is 
\[
  \indexhigher_{ \cT }( s ) = \sum_{ i=0 }^{ d } (-1)^i [t_i],
\]
viewed as an element of the split Grothendieck group $\Ksp( \cT )$.
\hfill $\Box$
\end{Definition}
\endgroup
The superscript $\pentagon$ serves to distinguish the $( d+2 )$-angulated index from the triangulated index we will also need; see Definition \ref{def:index}.  Remark \ref{rmk:index2} shows that $\indexhigher_{ \cT }( s )$ is well-defined and depends only on the isomorphism class of $s$.  Definition \ref{def:index2} is the natural $( d+2 )$-angulated generalisation of Palu's index on triangulated categories, see \cite[sec.\ 2.1]{Palu}, which arises as the special case $d = 1$.

Our first main result is that the $( d+2 )$-angulated index is additive on $( d+2 )$-angles up to an error term.  This generalises \cite[prop.\ 2.2]{Palu} to $( d+2 )$-angulated categories.  Set $\Gamma = \End_{ \cS }( t )$.  There is a functor $F : \cS \rightarrow \mod( \Gamma )$ given by $F( - ) = \cS( t,- )$.  In particular, a morphism $\gamma$ in $\cS$ induces a module $\Image F\gamma$ in $\mod( \Gamma )$, hence an element $[ \Image F\gamma ]$ in the Grothendieck group $\K0( \mod\,\Gamma )$.  
\begingroup
\renewcommand{\theLemma}{\Alph{Lemma}}
\begin{Theorem}
[=Theorem \ref{thm:index_additive5}]
\label{thm:index_additive5_intro}
There is a unique homomorphism $\theta : \K0( \mod\,\Gamma ) \rightarrow \Ksp( \cT )$ of abelian groups such that if $s_{ d+1 } \rightarrow \cdots \rightarrow s_0 \xrightarrow{ \gamma } \Sigma^d s_{ d+1 }$ is a $( d+2 )$-angle in $\cS$, then
\[
\tag*{$\Box$}
  \sum_{ i=0 }^{ d+1 } (-1)^i \indexhigher_{ \cT }( s_i ) = 
  \theta \big( [ \Image F\gamma ] \big).
\]
\end{Theorem}
\endgroup
Our second main result is that tropical friezes on $( d+2 )$-angulated categories can be constructed using the $( d+2 )$-angulated index.  This generalises \cite[thm.\ 3.1]{G} to $( d+2 )$-angulated categories.  The proof relies crucially on Theorem \ref{thm:index_additive5_intro}.
\begingroup
\renewcommand{\theLemma}{\Alph{Lemma}}
\begin{Theorem}
[=Theorem \ref{thm:frieze}]
\label{thm:frieze_intro}
Assume that $d$ is odd and $\cS$ is $2d$-Calabi--Yau.  Let $\varphi : \Ksp( \cT ) \rightarrow \BZ$ be a homomorphism of abelian groups.  The composition 
\[
  \varphi \circ \indexhigher_{ \cT } : \obj\, \cS \rightarrow \BZ
\]
is a tropical frieze on $\cS$ if $\varphi$ satisfies $\varphi\theta\big( [M] \big) \geqslant 0$ for each $M \in \mod( \Gamma )$.
\hfill $\Box$
\end{Theorem}
\endgroup
To prove Theorems \ref{thm:index_additive5_intro} and \ref{thm:frieze_intro} we will need the additional assumption that $\cS$ is ``standard'', in the sense that its structure as $( d+2 )$-angulated category is obtained from \cite[thm.\ 1]{GKO} by realising $\cS$ as a $d$-cluster tilting subcategory, stable under $\Sigma^d$, inside a $k$-linear $\Hom$-finite triangulated category with split idempotents.  See \cite[def.\ 5.15]{OT} which coined the term ``standard''.

The paper is organised as follows:  Section \ref{sec:prelim} is preliminary.  Section \ref{sec:F} investigates a triangulated version of the functor $F$ defined before Theorem \ref{thm:index_additive5_intro}.  Section \ref{sec:index} defines the triangulated index with respect to an $n$-cluster tilting object in the sense of \cite[sec.\ 5.1]{KR}; this is another generalisation of Palu's index on triangulated categories.  Section \ref{sec:theta} shows additivity with an error term of the triangulated index.  Section \ref{sec:higher} uses the triangulated index to prove Theorem \ref{thm:index_additive5_intro}.  Section \ref{sec:friezes} proves Theorem \ref{thm:frieze_intro}.  Section \ref{sec:example} shows the example mentioned at the end of Section \ref{subsec:frieze}.

\section{Preliminaries}
\label{sec:prelim}

This section provides a setup and collects some reminders on the papers \cite{I3}, \cite{IY}, and \cite{KR}.

\begin{Setup}
\label{set:blanket1}
In Sections \ref{sec:prelim} through \ref{sec:friezes} the following is fixed.
\begin{itemize}
\setlength\itemsep{4pt}

  \item  $k$ is a field.

  \item  $\cC$ is a $k$-linear $\Hom$-finite triangulated category with split idempotents and suspension functor $\Sigma$.  
  
  \item  $n \geqslant 2$ is an integer.
  
  \item  $t \in \cC$ is an $n$-cluster tilting object.
  
  \item  $\cT = \add( t )$ is the $n$-cluster tilting subcategory associated to $t$.

  \item  $\Gamma = \End_{ \cC }( t )$ is the endomorphism algebra of $t$.
  \hfill $\Box$

\end{itemize}
\end{Setup}

We recall the definitions of $n$-cluster tilting objects and subcategories from \cite[sec.\ 5.1]{KR}.  They are triangulated versions of \cite[def.\ 2.2]{I3}.

\begin{Definition}
\label{def:cluster_tilting_subcategory}
An {\em $n$-cluster tilting object of the triangulated category} $\cC$ is an object $t$ such that $\cT = \add( t )$ is an {\em $n$-cluster tilting subcategory}, that is, $\cT \subseteq \cC$ is a full subcategory which is functorially finite and satisfies
\[
  \cT = \{\, c \in \cC \,|\, \Ext_{ \cC }^{1..n-1}( \cT,c ) = 0 \,\} = \{\, c \in \cC \,|\, \Ext_{ \cC }^{1..n-1}( c,\cT ) = 0 \,\}.
\]
A $2$-cluster tilting object is simply called a {\em cluster tilting object}.
\hfill $\Box$
\end{Definition}

\begin{Remark}
An $n$-cluster tilting subcategory $\cT$ has the form $\add( t )$ if and only if it has finitely many indecomposable objects.  With small modifications, our results remain valid without this assumption.  Then $\mod( \Gamma )$ must be replaced with the category $\mod( \cT )$ of finitely presented contravariant $k$-linear functors $\cT \rightarrow \mod( k )$, and the functor $F : c \mapsto \cC( t,c )$ with values in $\mod( \Gamma )$, which will appear in Definition \ref{def:F}, must be replaced with the functor $c \mapsto \cC( -,c )|_{ \cT }$ with values in $\mod( \cT )$.
\hfill $\Box$
\end{Remark}

Recall the following definition and lemma from the start of \cite[sec.\ 2]{IY}.

\begin{Definition}
If $\cX, \cY \subseteq \cC$ are full subcategories, then there is a full subcategory
\[
\tag*{$\Box$}
  \cX * \cY = \{\, c \in \cC \,|\, \mbox{there is a triangle $x \rightarrow c \rightarrow y \rightarrow \Sigma x$ with $x \in \cX$, $y \in \cY$} \,\}.
\]
\end{Definition}

\begin{Lemma}
\label{lem:associativity}
The operation $*$ is associative, so the notation $\cX * \cY * \cZ$ makes sense without brackets.
\end{Lemma}

The following is \cite[def.\ 2.2]{IY}.

\begin{Definition}
\label{def:torsion_pair}
A {\em torsion pair} in $\cC$ is a pair $( \cX,\cY )$ of full subcategories $\cX, \cY  \subseteq \cC$ such that $\cC( \cX,\cY ) = 0$ and $\cC = \cX * \cY$.
\hfill $\Box$
\end{Definition}

Part (i) of the following lemma is \cite[thm.\ 3.1(2)]{IY}, and part (ii) is an immediate consequence of part (i).

\begin{Lemma}
\label{lem:Sl}
Let $0 \leqslant \ell \leqslant n-2$ be an integer.
\begin{enumerate}
\setlength\itemsep{4pt}

  \item  There is a torsion pair $( \cT * \cdots * \Sigma^{ \ell }\cT,\Sigma^{ \ell+1 }\cT * \cdots * \Sigma^{ n-1 }\cT )$ in $\cC$.

  \item  $\cT * \cdots * \Sigma^{ \ell }\cT$ is closed under extensions.
  
\end{enumerate}
\end{Lemma}

The following is \cite[thm.\ 3.1(1)]{IY}.

\begin{Lemma}
\label{lem:T-resolution0}
$\cC = \cT * \cdots * \Sigma^{ n-1 }\cT$.
\end{Lemma}

\section{The homological functor $F$}
\label{sec:F}

This section shows some properties of the following functor.

\begin{Definition}
\label{def:F}
Let $F$ be the homological functor
\[
\tag*{$\Box$}
  F : \cC \rightarrow \mod( \Gamma )
  \;\;,\;\;
  F( - ) = \cC( t,- ).
\]
\end{Definition}

The following lemma will be used several times.

\begin{Lemma}
\label{lem:vanishing}
$F( \Sigma \cT * \cdots * \Sigma^{ n-1 }\cT ) = 0$.
\end{Lemma}

\begin{proof}
The lemma means $\cC( t,\Sigma \cT * \cdots * \Sigma^{ n-1 }\cT ) = 0$, but we have $t \in \cT$ and there is a torsion pair $( \cT,\Sigma \cT * \cdots * \Sigma^{ n-1 }\cT )$ in $\cC$ by Lemma \ref{lem:Sl}(i).
\end{proof}

The following lemma is elementary.

\begin{Lemma}
\label{lem:index_additive1.5}
If $a \xrightarrow{} b \xrightarrow{ \beta } c \xrightarrow{ \gamma } \Sigma a$ is a triangle in $\cC$, then $F\gamma = 0$ if and only if $\cC( t',\beta )$ is surjective for each $t' \in \cT$.  
\end{Lemma}

The following lemma is a special case of \cite[prop.\ 6.2(3)]{IY}, where $[\Sigma \cT]$ denotes the ideal of morphisms in $\cT * \Sigma \cT$ which factor through an object of $\Sigma \cT$.

\begin{Lemma}
\label{lem:IY623}
The restriction of $F$ to $\cT * \Sigma \cT$ induces an equivalence of categories $( \cT * \Sigma \cT ) / [ \Sigma \cT ] \xrightarrow{ \sim } \mod( \Gamma )$. 
\end{Lemma}

\begin{Lemma}
\label{lem:map_induced_by_F}
For $m \in \cT * \Sigma \cT$ and $c \in \cC$, consider the map
\[
  \cC( m,c ) \xrightarrow{ F( - ) } \Hom_{ \Gamma }( Fm,Fc ).
\]
Then:
\begin{enumerate}
\setlength\itemsep{4pt}
  
  \item  $F( - )$ is surjective.

  \item  The kernel of $F( - )$ consists of the morphisms $m \rightarrow c$ which factor through an object of $\Sigma \cT$.  
\end{enumerate}
\end{Lemma}

\begin{proof}
Part (i) is a special case of \cite[lem.\ 1.11(ii)]{HJ}.

For part (ii), suppose that $m \xrightarrow{ \mu } c$ factors as $m \rightarrow \Sigma t_1 \rightarrow c$ for some $t_1 \in \cT$.  Then $F\mu = 0$ because $F( \Sigma t_1 ) = 0$ by Lemma \ref{lem:vanishing}.

Conversely, suppose that $m \xrightarrow{ \mu } c$ satisfies $F\mu = \cC( t,\mu ) = 0$.  Then $\cC( t_0,\mu ) = 0$ for each $t_0 \in \cT$.  There is a triangle $t_1 \rightarrow t_0 \xrightarrow{ \tau'_0 } m \xrightarrow{ \mu' } \Sigma t_1$ with $t_0, t_1 \in \cT$ since $m \in \cT * \Sigma \cT$.  Since $\cC( t_0,\mu ) = 0$ we have $\mu\tau'_0 = 0$, so $\mu$ factors as $m \xrightarrow{ \mu' } \Sigma t_1 \rightarrow c$.
\end{proof}

\begin{Lemma}
\label{lem:image_induced_by_F}
For each morphism $k \xrightarrow{ \kappa } a$ in $\cC$ with $k \in \cT * \Sigma \cT$, there is a commutative diagram
\begingroup
\[
  \vcenter{
  \xymatrix 
  {
    k \ar[rr]^{ \kappa } \ar[rd]_{ \sigma } && a \\
    & m \ar[ur]_{ \iota } &
                        }
          }
\]
\endgroup
with $m \in \cT * \Sigma \cT$, where $F\sigma$ is surjective, $F\iota$ injective.  Hence $Fm$ can be identified with $\Image F\kappa$. 
\end{Lemma}

\begin{proof}
There is a canonical factorisation
\begingroup
\[
  \vcenter{
  \xymatrix 
  {
    Fk \ar[rr]^{ F\kappa } \ar@{->>}[rd]_{ \rho } && Fa. \\
    & \Image F\kappa \ar@{^(->}[ur]_{ \theta } &
                        }
          }
\]
\endgroup
By Lemma \ref{lem:IY623} there exists $m' \in \cT * \Sigma \cT$ such that $Fm' \cong \Image F\kappa$, and by Lemma \ref{lem:map_induced_by_F}(i) there are $k \xrightarrow{ \sigma' } m' \xrightarrow{ \iota' } a$ such that
\begingroup
\begin{equation}
\label{equ:image_induced_by_F1}
  \vcenter{
  \xymatrix @+0.35pc 
  {
    Fk \ar[rr]^{ F\kappa } \ar@{->>}[rd]_<<<<<{ F\sigma' } && Fa \\
    & Fm' \ar@{^(->}[ur]_>>>>>>{ F\iota' } &
                        }
          }
\end{equation}
\endgroup
can be identified with the previous diagram.  Then $F( \kappa - \iota'\sigma' ) = F\kappa - \theta\rho = 0$, so by Lemma \ref{lem:map_induced_by_F}(ii) the morphism $\kappa - \iota'\sigma'$ factors as $k \xrightarrow{ \sigma'' } \Sigma t'' \xrightarrow{ \iota'' } a$ for some $t'' \in \cT$, whence $\kappa = \iota'\sigma' + \iota''\sigma''$.

This means that the following diagram is commutative.
\begingroup
\begin{equation}
\label{equ:image_induced_by_F2}
  \vcenter{
  \xymatrix 
  {
    k \ar[rr]^{ \kappa } \ar[rd]_<<<<<{ \begin{psmallmatrix} \sigma' \\ \sigma'' \end{psmallmatrix} } && a \\
    & m' \oplus \Sigma t'' \ar[ur]_<<<<<<<<{ ( \iota' \; \iota'' ) } &
                        }
          }
\end{equation}
\endgroup
It is clear that $m' \oplus \Sigma t'' \in \cT * \Sigma \cT$, and applying $F$ to the diagram \eqref{equ:image_induced_by_F2} recovers \eqref{equ:image_induced_by_F1} because $F( \Sigma t'' ) = 0$ by Lemma \ref{lem:vanishing}, so $F\begin{psmallmatrix} \sigma' \\ \sigma'' \end{psmallmatrix}$ is surjective, $F( \iota' \; \iota'' )$ injective.  Hence \eqref{equ:image_induced_by_F2} can be used as the diagram in the lemma. 
\end{proof}

The following lemma generalises \cite[lem.\ 3.1]{Palu} to the case where $t$ is an $n$-cluster tilting object.  The proof remains the same, but we include it for the convenience of the reader.

\begin{Lemma}
\label{lem:short_exact_sequence}
Each short exact sequence in $\mod( \Gamma )$ has the form $0 \rightarrow Fk \xrightarrow{ F\kappa } F\ell \xrightarrow{ F\lambda } Fm \rightarrow 0$ up to isomorphism, where
\begin{equation}
\label{equ:short_exact_sequence1}
  k \xrightarrow{ \kappa } \ell \xrightarrow{ \lambda } m \xrightarrow{} \Sigma k
\end{equation}
is a triangle in $\cC$ with $k,\ell,m \in \cT * \Sigma \cT$.  
\end{Lemma}

\begin{proof}
Given a short exact sequence $0 \rightarrow K \rightarrow L \rightarrow M \rightarrow 0$ in $\mod( \Gamma )$, Lemma \ref{lem:IY623} implies that there is a morphism $k \xrightarrow{ \kappa' } \ell'$ in $\cT * \Sigma \cT$ such that $Fk \xrightarrow{ F\kappa' } F\ell'$ can be identified with $K \rightarrow L$.  Pick a triangle $t_1 \rightarrow t_0 \rightarrow k \xrightarrow{ \delta } \Sigma t_1$ in $\cC$ with $t_i \in \cT$.  The octahedral axiom gives the following diagram where each row and column is a triangle.
\[
  \xymatrix @-1.6pc @! {
    \Sigma^{ -1 }k \ar[r] \ar[d] & t_1 \ar[r] \ar[d] & t_0 \ar[r] \ar[d] & k \ar[d] \\
    0 \ar[r] \ar[d] & \ell' \ar@{=}[r] \ar[d] & \ell' \ar[r] \ar[d] & 0 \ar[d] \\
    k \ar[r]^-{ \begin{psmallmatrix} \kappa' \\ \delta \end{psmallmatrix} } \ar@{=}[d] & \ell' \oplus \Sigma t_1 \ar[r]^-{ \lambda } \ar[d]_{ ( 0 \; \id ) } & m \ar[r]^{ \mu } \ar[d] & \Sigma k \ar@{=}[d] \\
    k \ar[r]_-{ \delta } & \Sigma t_1 \ar[r] & \Sigma t_0 \ar[r] & \Sigma k \\
                      }
\]
We claim that the third row of the diagram can be used as the triangle \eqref{equ:short_exact_sequence1}:  We have $F( \Sigma t_1 ) = 0$ by Lemma \ref{lem:vanishing}, so the morphism $Fk \xrightarrow{ F\begin{psmallmatrix} \kappa' \\ \delta \end{psmallmatrix} } F( \ell' \oplus \Sigma t_1 )$ can also be identified with $K \rightarrow L$.  Similarly, $F( \Sigma t_0 ) = 0$ while $\mu$ factors through $\Sigma t_0$, so $F\mu = 0$ whence $F\lambda$ is surjective.  Finally, $\ell' \in \cT * \Sigma \cT$ and $\Sigma t_i \in \Sigma \cT$, so  $\ell' \oplus \Sigma t_1 \in \cT * \Sigma \cT$ is clear, and the third column of the diagram shows $m \in ( \cT * \Sigma \cT ) * \Sigma \cT = \cT * \Sigma \cT$.  The last equality is by Lemma \ref{lem:associativity} and the easy observation $\Sigma \cT * \Sigma \cT = \Sigma \cT$.  
\end{proof}

\section{The triangulated index with respect to an $n$-cluster tilting object}
\label{sec:index}

This section introduces the triangulated index with respect to an $n$-cluster tilting object, see Definition \ref{def:index}, and shows some basic properties.  Like the $( d+2 )$-angulated index, it is a generalisation of Palu's index on triangulated categories, see \cite[sec.\ 2.1]{Palu}, which arises as the special case $n = 2$.

\begin{Definition}
\label{def:tower}
A {\em tower of triangles} in $\cC$ is a diagram of the form
\[
  \xymatrix @-1.7pc @! {
    & c_{ m-1 } \ar[rr] \ar[dr] & & c_{ m-2 } \ar[rr] \ar[dr] & & \cdots \ar[rr] & & c_2 \ar[rr] \ar[dr] & & c_1 \ar[dr] \\
    c_m \ar[ur] && v_{ m-1.5 } \ar[ur] \ar@{~>}[ll] & & v_{ m-2.5 } \ar@{~>}[ll] & \cdots & v_{ 2.5 } \ar[ur] & & v_{ 1.5 } \ar[ur] \ar@{~>}[ll] && c_0 \lefteqn{,} \ar@{~>}[ll] \\
               }
\]
where $m \geqslant 2$ is an integer, a wavy arrow $\xymatrix { x \ar@{~>}[r] & y }$ signifies a morphism $\xymatrix { x \ar[r] & \Sigma y, }$ each oriented triangle is a triangle in $\cC$, and each non-oriented triangle is commutative.

If $m = 2$ then the tower is a single triangle $c_2 \rightarrow c_1 \rightarrow c_0 \rightarrow \Sigma c_0$, and there are no objects $v_{ \ast }$.  
\hfill $\Box$
\end{Definition}

\begin{Lemma}
\label{lem:T-resolution1}
For $c \in \cC$ there is a tower of triangles in $\cC$,
\[
  \xymatrix @-1.4pc @! {
    & t_{ n-2 } \ar[rr] \ar[dr]^{ \tau_{ n-2 } } & & t_{ n-3 } \ar[rr] \ar[dr]^{ \tau_{ n-3 } } & & \cdots \ar[rr] & & t_1 \ar[rr] \ar[dr]^{ \tau_1 } & & t_0 \ar[dr]^{ \tau_0 } \\
    t_{ n-1 } \ar[ur] && v_{ n-2.5 } \ar[ur] \ar@{~>}[ll] & & v_{ n-3.5 } \ar@{~>}[ll] & \cdots & v_{ 1.5 } \ar[ur] & & v_{ 0.5 } \ar[ur] \ar@{~>}[ll] && c \lefteqn{,} \ar@{~>}[ll] \\
               }
\]
where each $\tau_i$ is a $\cT$-cover; in particular, $t_i \in \cT$ for each $i$.  The $t_i$ are determined up to isomorphism.
\end{Lemma}

\begin{proof}
The tower exists by \cite[cor.\ 3.3 and its proof]{IY}.  The $t_i$ are determined up to isomorphism by the construction given there, because $c$ or $v_{i - 0.5}$ determines the $\cT$-cover $\tau_i$ up to isomorphism, which in turn determines the cone $v_{ i+0.5 }$ or $t_{ n-1 }$ up to isomorphism.
\end{proof}

\begin{Definition}
\label{def:index}
The {\em triangulated index of $c \in \cC$ with respect to $t$} is the element 
\[
  \index_{ \cT }( c ) = \sum_{ i=0 }^{ n-1 } (-1)^i [t_i]
\] 
of the split Grothendieck group $\Ksp( \cT )$, where the $t_i$ come from the tower in Lemma \ref{lem:T-resolution1}.  
\hfill $\Box$
\end{Definition}

\begin{Lemma}
\label{lem:index_additive0}
If $t' \in \cT$ is an object, $0 \leqslant \ell \leqslant n-1$ an integer, then $\index_{ \cT }( \Sigma^{ \ell }t' ) = (-1)^{ \ell }[t']$.
\end{Lemma}

\begin{proof}
There is a diagram as in Lemma \ref{lem:T-resolution1} with $t_{ \ell } = t'$ and every other $t_i$ equal to zero. 
\end{proof}

The following three lemmas show some instances of additivity of the triangulated index with respect to an $n$-cluster tilting object.  They will be used to prove Theorem \ref{thm:index_additive4}, of which they are special cases.

\begin{Lemma}
\label{lem:index_additive1}
If $c,c' \in \cC$ then $\index_{ \cT }( c \oplus c' ) = \index_{ \cT }( c ) + \index_{ \cT }( c' )$.
\end{Lemma}

\begin{proof}
For each of $c$ and $c'$, there is a diagram as in Lemma \ref{lem:T-resolution1}.  Taking their direct sum in an obvious sense produces a similar diagram for $c \oplus c'$.
\end{proof}

\begin{Lemma}
\label{lem:index_additive2}
If
\begin{equation}
\label{equ:index_additive2a}
  \vcenter{
  \xymatrix @C=1pc @! {
    c' \ar[r] & t_c \ar[r]^{ \tau_c } & c \ar[r]^-{ \delta_c } & \Sigma ( c' ) \\
                                 }
          }
\end{equation}  
is a triangle in $\cC$ with $t_c \in \cT$ and $F\delta_c = 0$, then $[t_c] = \index_{ \cT }( c' ) + \index_{ \cT }( c )$.
\end{Lemma}

\begin{proof}
Combining Lemma \ref{lem:index_additive1.5} with the condition $t_c \in \cT$ shows that $\tau_c$ is a $\cT$-precover.  Hence up to isomorphism, \eqref{equ:index_additive2a} is the direct sum of two triangles
\[
  \xymatrix @R=-1.75pc @C=0pc @! {
    v_{ 0.5 } \ar[r] & t_0 \ar[r]^{ \tau_0 } & c \ar[r] & \Sigma( v_{ 0.5 } ), \\
    t'_0 \ar[r]^{ \id } & t'_0 \ar[r] & 0 \ar[r] & \Sigma( t'_0 )
                       }
\]
where $t_0 \xrightarrow{ \tau_0 } c$ is a $\cT$-cover and $t'_0 \in \cT$.  In particular, $c' \cong v_{ 0.5 } \oplus t'_0$ and $t_c \cong t_0 \oplus t'_0$.  Lemmas \ref{lem:index_additive1} and \ref{lem:index_additive0} show
\begin{equation}
\label{equ:index_additive2b}
  \index_{ \cT }( c' )
  = \index_{ \cT }( v_{ 0.5 } ) + \index_{ \cT }( t_c ) - \index_{ \cT }( t_0 )
  = \index_{ \cT }( v_{ 0.5 } ) + [t_c] - [t_0].
\end{equation}
The triangle $v_{ 0.5 } \xrightarrow{} t_0 \xrightarrow{ \tau_0 } c \xrightarrow{} \Sigma( v_{ 0.5 } )$ is part of a diagram as in Lemma \ref{lem:T-resolution1}, which shows $\index_{ \cT }( c ) = \sum_{ i=0 }^{ n-1 } (-1)^i [t_i]$ and $\index_{ \cT }( v_{ 0.5 } ) = - \sum_{ i=1 }^{ n-1 } (-1)^i [t_i]$, whence
\[
  \index_{ \cT }( c ) = [t_0] - \index_{ \cT }( v_{ 0.5 } ).
\]
Combining with Equation \eqref{equ:index_additive2b} proves the lemma.
\end{proof}

\begin{Lemma}
\label{lem:index_additive3}
If
\begin{equation}
\label{equ:index_additive3a}
  a \xrightarrow{ \alpha } b \xrightarrow{ \beta } c \xrightarrow{ \gamma } \Sigma a
\end{equation}
is a triangle in $\cC$ with $F\gamma = 0$, then
\begin{equation}
\label{equ:index_additive3b}
  \index_{ \cT }( b ) = \index_{ \cT }( a ) + \index_{ \cT }( c ).
\end{equation}
\end{Lemma}

\begin{proof}
Let $0 \leqslant \ell \leqslant n-1$ be an integer.  In addition to the assumptions in the lemma, suppose $a,c \in \cT * \cdots * \Sigma^{ \ell }\cT$.  We will prove Equation \eqref{equ:index_additive3b} by induction on $\ell$.  The lemma will follow because of Lemma \ref{lem:T-resolution0}.

If $\ell = 0$ then $a,c \in \cT$.  Since $\cC( \cT,\Sigma \cT ) = 0$ we have $c \xrightarrow{ \gamma } \Sigma a$ equal to zero, so the triangle \eqref{equ:index_additive3a} is split and $b \cong a \oplus c$.  Hence Equation \eqref{equ:index_additive3b} follows from Lemma \ref{lem:index_additive1}.

If $1 \leqslant \ell \leqslant n-1$, then by definition of $\cT * \cdots * \Sigma^{ \ell }\cT$ there are triangles
\begin{equation}
\label{equ:index_additive3z}
  a' \xrightarrow{} t_a \xrightarrow{ \tau_a } a \xrightarrow{ \delta_a } \Sigma a'
  \;\;,\;\;
  c' \xrightarrow{} t_c \xrightarrow{ \tau_c } c \xrightarrow{ \delta_c } \Sigma c'
\end{equation}
with $t_a,t_c \in \cT$ and
\begin{equation}
\label{equ:index_additive3d}
  a',c' \in \cT * \cdots * \Sigma^{ \ell-1 }\cT.
\end{equation}
We know $F\gamma = 0$, so Lemma \ref{lem:index_additive1.5} says that  there is $t_c \xrightarrow{ \tau } b$ such that $\tau_c = \beta\tau$.  This gives the following commutative diagram of triangles.
\[
  \xymatrix @-1.2pc @! {
    t_a \ar[r]^-{ \begin{psmallmatrix} \id \\ 0 \end{psmallmatrix} } \ar[d]_{ \tau_a } & t_a \oplus t_c \ar[r]^-{ ( 0 \; \id ) } \ar[d]_{ ( \alpha\tau_a \; \tau ) } & t_c \ar[r]^0 \ar[d]^{ \tau_c } & \Sigma t_a \ar[d]^{ \Sigma \tau_a } \\
    a \ar[r]_{ \alpha } & b \ar[r]_{ \beta } & c \ar[r]_{ \gamma } & \Sigma a \\
                      }
\]
Since the first triangle is split, by \cite[lem.\ 1.7, def.\ 1.9, and thm.\ 2.3]{N} the diagram can be completed to the following diagram, where each row and column is a triangle and each square is commutative, except for the one at the bottom right which is anticommutative.
\begin{equation}
\label{equ:index_additive3_diagram}
\vcenter{
  \xymatrix @-1.2pc @! {
    a' \ar[r] \ar[d] & b' \ar[r] \ar[d] & c' \ar[r]^{ \gamma' } \ar[d] & \Sigma a' \ar[d] \\
    t_a \ar[r]^-{ \begin{psmallmatrix} \id \\ 0 \end{psmallmatrix} } \ar[d]_{ \tau_a } & t_a \oplus t_c \ar[r]^-{ ( 0 \; \id ) } \ar[d]_{ ( \alpha\tau_a \; \tau ) } & t_c \ar[r]^0 \ar[d]^{ \tau_c } & \Sigma t_a \ar[d]^{ \Sigma \tau_a } \\
    a \ar[r]_{ \alpha } \ar[d] & b \ar[r]_{ \beta } \ar[d]_{ \delta_b } & c \ar[r]_{ \gamma } \ar[d] & \Sigma a \ar[d] \\
    \Sigma a' \ar[r] & \Sigma b' \ar[r] & \Sigma c' \ar[r] & \Sigma^2 a' \\
                      }
        }                      
\end{equation}
The second column is the triangle
\begin{equation}
\label{equ:index_additive3y}
  b' \xrightarrow{} t_a \oplus t_c \xrightarrow{ ( \alpha \tau_a \; \tau ) } b \xrightarrow{ \delta_b } \Sigma b'.  	
\end{equation}

Equation \eqref{equ:index_additive3d} implies
\begin{equation}
\label{equ:index_additive3d.5}
  F( \Sigma a' ) = 0
\end{equation}
by Lemma \ref{lem:vanishing}.  In particular, the first of the triangles \eqref{equ:index_additive3z} satisfies $F\delta_a = 0$, so Lemma \ref{lem:index_additive2} gives
\begin{equation}
\label{equ:index_additive3d.8}
  [t_a] = \index_{ \cT }( a' ) + \index_{ \cT }( a ).
\end{equation}
Similarly,
\begin{equation}
\label{equ:index_additive3d.9}
  [t_c] = \index_{ \cT }( c' ) + \index_{ \cT }( c ).
\end{equation}
The first row of diagram \eqref{equ:index_additive3_diagram} combined with Equation \eqref{equ:index_additive3d} and Lemma \ref{lem:Sl}(ii) gives $b' \in \cT * \cdots * \Sigma^{ \ell-1 }\cT$ whence $F( \Sigma b' ) = 0$ by Lemma \ref{lem:vanishing}.  In particular, the triangle \eqref{equ:index_additive3y} satisfies $F\delta_b = 0$, so Lemma \ref{lem:index_additive2} gives
\begin{equation}
\label{equ:index_additive3h}
  [t_a \oplus t_c] = \index_{ \cT }( b' ) + \index_{ \cT }( b ).
\end{equation}

Finally, Equation \eqref{equ:index_additive3d.5} implies $F\gamma' = 0$.  By Equation \eqref{equ:index_additive3d}, this means that induction applies to the triangle $a' \xrightarrow{} b' \xrightarrow{} c' \xrightarrow{ \gamma' } \Sigma a'$, whence
\[
  \index_{ \cT }( b' ) = \index_{ \cT }( a' ) + \index_{ \cT }( c' ).
\]
Combined with Equations \eqref{equ:index_additive3d.8}, \eqref{equ:index_additive3d.9}, and \eqref{equ:index_additive3h}, this proves Equation \eqref{equ:index_additive3b}. 
\end{proof}

\section{Additivity of the triangulated index up to the error term $\theta$}
\label{sec:theta}

This section shows that the triangulated index with respect to an $n$-cluster tilting object is additive up to an error term given by a homomorphism $\theta$.  The following lemma is due to \cite[sec.\ 3.1]{Palu} in the case where $t$ is a $2$-cluster tilting object.

\begin{Lemma}
\label{lem:theta}
There is a unique homomorphism
\[
  \theta : \K0( \mod\,\Gamma ) \rightarrow \Ksp( \cT )
\]
defined by
\[
  \theta( [Fm] ) = \index_{ \cT }( \Sigma^{ -1 }m ) + \index_{ \cT }( m )
\]
for $m \in \cT * \Sigma \cT$.
\end{Lemma}

\begin{proof}
Each $M \in \mod( \Gamma )$ has the form $Fm$ for some $m \in \cT * \Sigma \cT$ by Lemma \ref{lem:IY623}, so the lemma follows from Lemmas \ref{lem:theta1} and \ref{lem:theta2}.
\end{proof}

Lemmas \ref{lem:theta1} and \ref{lem:theta2} generalise \cite[lem.\ 2.1(4) and lem.\ 1.3]{Palu} to the case where $t$ is an $n$-cluster tilting object.  The proofs remain the same, but we include them for the convenience of the reader.

\begin{Lemma}
\label{lem:theta1}
If $m',m'' \in \cT * \Sigma \cT$ satisfy $Fm' \cong Fm''$, then
\[
  \index_{ \cT }( \Sigma^{ -1 }m' ) + \index_{ \cT }( m' ) =
  \index_{ \cT }( \Sigma^{ -1 }m'' ) + \index_{ \cT }( m'' ).
\]
\end{Lemma}

\begin{proof}
Lemma \ref{lem:IY623} implies that $m'$ and $m''$ differ only by summands in $\Sigma \cT$, so there are $t',t'' \in \cT$ such that
\begin{equation}
\label{equ:theta1}
  m' \oplus \Sigma t' \cong m'' \oplus \Sigma t''.  
\end{equation}
We have
\begin{align*}
  & \index_{ \cT }( \Sigma^{ -1 }m' ) + \index_{ \cT }( m' ) \\
  & \;\; = \index_{ \cT }( \Sigma^{ -1 }m' ) + [t'] + \index_{ \cT }( m' ) - [t'] \\
  & \;\; \stackrel{ \rm (a) }{ = } \index_{ \cT }( \Sigma^{ -1 }m' ) + \index_{ \cT }( t' ) + \index_{ \cT }( m' ) + \index_{ \cT }( \Sigma t' ) \\
  & \;\; \stackrel{ \rm (b) }{ = } \index_{ \cT }\big( \Sigma^{ -1 }( m' \oplus \Sigma t' ) \big)
  + \index_{ \cT }( m' \oplus \Sigma t' ),
\end{align*}
where (a) and (b) are by Lemmas \ref{lem:index_additive0} and \ref{lem:index_additive1}.  There is a similar formula with $m''$, $t''$ instead of $m'$, $t'$.  By Equation \eqref{equ:theta1} the two formulae have the same right hand sides, so the left hand sides are equal, completing the proof.
\end{proof}

\begin{Lemma}
\label{lem:theta2}
If
\begin{equation}
\label{equ:theta2_1}
  0 \rightarrow Fk \rightarrow F\ell \rightarrow Fm \rightarrow 0
\end{equation}
is a short exact sequence in $\mod( \Gamma )$ with $k,\ell,m \in \cT * \Sigma \cT$, then
\[
  \index_{ \cT }( \Sigma^{ -1 }\ell ) + \index_{ \cT }( \ell )
  = \index_{ \cT }( \Sigma^{ -1 }k ) + \index_{ \cT }( k )
    + \index_{ \cT }( \Sigma^{ -1 }m ) + \index_{ \cT }( m ).
\]
\end{Lemma}

\begin{proof}
By Lemma \ref{lem:short_exact_sequence} there is a triangle
\begin{equation}
\label{equ:theta2_2}
  k' \xrightarrow{ \kappa' } \ell' \xrightarrow{ \lambda' } m' \xrightarrow{ \mu' } \Sigma k'
\end{equation}
in $\cC$ with $k',\ell',m' \in \cT * \Sigma \cT$ such that the short exact sequence \eqref{equ:theta2_1} can be identified with
\[
  0 \rightarrow Fk' \xrightarrow{ F\kappa' } F\ell' \xrightarrow{ F\lambda' } Fm' \rightarrow 0.
\]
Lemma \ref{lem:theta1} implies
\begin{equation}
\label{equ:theta2_3}
  \index_{ \cT }( \Sigma^{ -1 }k ) + \index_{ \cT }( k )
  = \index_{ \cT }( \Sigma^{ -1 }k' ) + \index_{ \cT }( k' ),
\end{equation}
and similar equations for $\ell$ and $m$.  Moreover, $F\lambda'$ is surjective, $F\kappa'$ injective, so $F\mu' = F( \Sigma^{ -1 }\mu' ) = 0$.  Hence Lemma \ref{lem:index_additive3} can be applied to the triangle \eqref{equ:theta2_2} and its desuspension
\[
  \Sigma^{ -1 }k' \xrightarrow{ -\Sigma^{ -1 }\kappa' } \Sigma^{ -1 }\ell' \xrightarrow{ -\Sigma^{ -1 }\lambda' } \Sigma^{ -1 }m' \xrightarrow{ -\Sigma^{ -1 }\mu' } k',
\]
giving
\begin{align*}
  \index_{ \cT }( \ell' ) & = \index_{ \cT }( k' ) + \index_{ \cT }( m' ), \\
  \index_{ \cT }( \Sigma^{ -1 }\ell' ) & = \index_{ \cT }( \Sigma^{ -1 }k' ) + \index_{ \cT }( \Sigma^{ -1 }m' ).
\end{align*}
Adding these and combining with Equation \eqref{equ:theta2_3} and similar equations for $\ell$ and $m$ proves the lemma.
\end{proof}

The following theorem generalises a result from \cite{Palu} to the case where $t$ is an $n$-cluster tilting object, see \cite[prop.\ 2.2 and sec.\ 3.1]{Palu}.  The first half of the proof (``special case'') is the same as the proof of \cite[prop.\ 2.2]{Palu}.

\begin{Theorem}
\label{thm:index_additive4}
If
\begin{equation}
\label{equ:index_additive4a}
  a \xrightarrow{ \alpha } b \xrightarrow{ \beta } c \xrightarrow{ \gamma } \Sigma a
\end{equation}
is a triangle in $\cC$, then
\[
  \index_{ \cT }( b )
  = \index_{ \cT }( a ) + \index_{ \cT }( c )
  - \theta \big( [ \Image F\gamma ] \big).
\]
\end{Theorem}

\begin{proof}
Special case: $c \in \cT * \Sigma \cT$.  Lemma \ref{lem:image_induced_by_F} provides a commutative diagram
\begingroup
\[
  \vcenter{
  \xymatrix 
  {
    c \ar[rr]^-{ \gamma } \ar[rd]_{ \sigma } && \Sigma a \\
    & m \ar[ur]_{ \iota } &
                        }
          }
\]
\endgroup
where $m \in \cT * \Sigma \cT$, and where
\begin{equation}
\label{equ:index_additive4c}
  \mbox{$F\sigma$ is surjective, $F\iota$ injective.}
\end{equation}
Using this and the octahedral axiom gives a diagram where each row and column is a triangle.
\[
  \xymatrix @-0.4pc @! {
    0 \ar[r] \ar[d] & e \ar@{=}[r] \ar[d] & e \ar[r] \ar[d] & 0 \ar[d] \\
    a \ar[r] \ar@{=}[d] & b \ar[r] \ar[d] & c \ar[r]^{ \gamma } \ar[d]^{ \sigma } & \Sigma a \ar@{=}[d] \\
    a \ar[r] \ar[d] & b' \ar[r]_{ \rho } \ar[d]_{ \pi\rho } & m \ar[r]_{ \iota } \ar[d]^{ \pi } & \Sigma a \ar[d] \\
    0 \ar[r] & \Sigma e \ar@{=}[r] & \Sigma e \ar[r] & 0 \\
                      }
\]
Equation \eqref{equ:index_additive4c} implies $F\pi = F\rho = 0$ whence also $F( \pi\rho ) = 0$, so Lemma \ref{lem:index_additive3} gives
\begin{align*}
  \index_{ \cT }( a ) & = \index_{ \cT }( \Sigma^{ -1 }m ) + \index_{ \cT }( b' ), \\
  \index_{ \cT }( b ) & = \index_{ \cT }( e ) + \index_{ \cT }( b' ), \\
  \index_{ \cT }( c ) & = \index_{ \cT }( e ) + \index_{ \cT }( m ).
\end{align*}
Combining these with Lemma \ref{lem:theta} and observing that $Fm \cong \Image F\gamma$ proves the theorem. 

General case: $c \in \cC$.  If $n = 2$, then $\cC = \cT * \Sigma \cT$ by Lemma \ref{lem:T-resolution0}, so we are done by the special case above.  Assume $n \geqslant 3$.  By Lemma \ref{lem:Sl}(i) there is a triangle $c' \xrightarrow{ \chi' } c \rightarrow f \rightarrow \Sigma c'$ with $c' \in \cT * \Sigma \cT$ and
\begin{equation}
\label{equ:index_additive4c.5}	
  f \in \Sigma^2 \cT * \cdots * \Sigma^{ n-1 }\cT.
\end{equation}  
Using this, the triangle \eqref{equ:index_additive4a}, and the octahedral axiom gives a diagram where each row and column is a triangle.
\[
  \xymatrix @-1.5pc @! {
    0 \ar[r] \ar[d] & \Sigma^{ -1 }f \ar@{=}[r] \ar[d] & \Sigma^{ -1 }f \ar[r] \ar[d] & 0 \ar[d] \\
    a \ar[r] \ar@{=}[d] & b' \ar[r] \ar[d] & c' \ar[r]^{ \gamma' } \ar[d]^{ \chi' } & \Sigma a \ar@{=}[d] \\
    a \ar[r] \ar[d] & b \ar[r] \ar[d]_{ \varphi } & c \ar[r]_{ \gamma } \ar[d]^{ \psi } & \Sigma a \ar[d] \\
    0 \ar[r] & f \ar@{=}[r] & f \ar[r] & 0 \\
                      }
\]

The special case from the start of the proof applied to the second row gives
\begin{equation}
\label{equ:index_additive4d}	
  \index_{ \cT }( b' )
  = \index_{ \cT }( a ) + \index_{ \cT }( c' )
  - \theta( [ \Image F\gamma' ] ).
\end{equation}
Equation \eqref{equ:index_additive4c.5} and Lemma \ref{lem:vanishing} imply $Ff = 0$ whence $F\varphi = F\psi = 0$, so Lemma \ref{lem:index_additive3} gives
\begin{align}
\label{equ:index_additive4e}	
  \index_{ \cT }( b' ) & = \index_{ \cT }( \Sigma^{ -1 }f ) + \index_{ \cT }( b ), \\
\label{equ:index_additive4f}	  
  \index_{ \cT }( c' ) & = \index_{ \cT }( \Sigma^{ -1 }f ) + \index_{ \cT }( c ).
\end{align}
Finally, Equation \eqref{equ:index_additive4c.5} and Lemma \ref{lem:vanishing} also imply $F( \Sigma^{ -1 }f ) = 0$, so applying $F$ to the last two columns of the diagram gives
\[
  \xymatrix @-1.1pc @! {
    0 \ar[r] \ar[d] & 0 \ar[d] \\
    Fc' \ar[r]^-{ F\gamma' } \ar[d]_{ F\chi' }^{ \mbox{\rotatebox{90}{$\sim$}} } & F( \Sigma a ) \ar@{=}[d] \\
    Fc \ar[r]_-{ F\gamma } \ar[d] & F( \Sigma a ) \ar[d] \\
    0 \ar[r] & 0 \lefteqn{,} \\
                      }
\]
whence $\Image F\gamma' = \Image F\gamma$.  Combining with Equations \eqref{equ:index_additive4d}, \eqref{equ:index_additive4e}, and \eqref{equ:index_additive4f} proves the theorem.  
\end{proof}

\section{The $( d+2 )$-angulated index}
\label{sec:higher}

This section Proves Theorem \ref{thm:index_additive5_intro} from the introduction (= Theorem \ref{thm:index_additive5}).

\begin{Setup}
\label{set:blanket2}
In Sections \ref{sec:higher} and \ref{sec:friezes} the following is fixed:  Setup \ref{set:blanket1} still applies; in particular, we still have a triangulated category $\cC$, an integer $n \geqslant 2$, and an $n$-cluster tilting object $t \in \cC$, and we write $\cT = \add( t )$.  Moreover, $n$ is assumed to be even, we write $n = 2d$, and:
\begin{enumerate}
\setlength\itemsep{4pt}
 
  \item  There is a $d$-cluster tilting subcategory $\cS \subseteq \cC$, which satisfies $\Sigma^d \cS = \cS$.  
  
  \item  $t \in \cS$.
\hfill $\Box$

\end{enumerate}
\end{Setup}

\begin{Remark}
\label{rmk:blanket2}
Setup \ref{set:blanket2} implies:
\begin{enumerate}
\setlength\itemsep{4pt}
 
  \item  $\cS$ has a ``standard'' structure as $( d+2 )$-angulated category by \cite[thm.\ 1]{GKO}, with $d$-suspension functor $\Sigma^d$, the restriction to $\cS$ of the $d$'th power of the suspension functor $\Sigma$ of $\cC$.
  
  \item  $t \in \cS$ is an Oppermann--Thomas cluster tilting object by \cite[thm.\ 5.26]{OT}.
\hfill $\Box$
  
\end{enumerate}
\end{Remark}

We recall the definition of Oppermann--Thomas cluster tilting objects from \cite[def.\ 5.3]{OT}.

\begin{Definition}
\label{def:OT_cluster_tilting}
An object $t \in \cS$ is an {\em Oppermann--Thomas cluster tilting object} if it satisfies the following:
\begin{enumerate}
\setlength\itemsep{4pt}

  \item  $\cS( t,\Sigma^d t ) = 0$.

  \item  For each $s \in \cS$ there is a $( d+2 )$-angle $t_d \rightarrow \cdots \rightarrow t_0 \rightarrow s \rightarrow \Sigma^d t_d$ in $\cS$ with $t_i \in \add( t )$.
\hfill $\Box$

\end{enumerate}
\end{Definition}

\begin{Remark}
\label{rmk:index2}
Recall $\indexhigher_{ \cT }( s )$ from Definition \ref{def:index2} in the introduction.  Let us show that it is well-defined and depends only on the isomorphism class of $s \in \cS$:  The $( d+2 )$-angle \eqref{equ:index2_higher_angle} in Definition \ref{def:index2} exists by Definition \ref{def:OT_cluster_tilting}(ii).  That definition does not include the condition that the $\tau_i$ are in the radical of $\cS$, but this can be achieved by dropping trivial summands from the $\tau_i$.  Moreover, \eqref{equ:index2_higher_angle} is unique up to isomorphism: First note that $t_0 \rightarrow s$ is unique up to isomorphism because it is a $\cT$-cover by \cite[lem.\ 3.12(a)]{F}, then apply \cite[lem.\ 5.18(2)]{OT}, which is valid under the assumptions we have made on $\cS$.
\hfill $\Box$
\end{Remark}

\begin{Proposition}
\label{pro:indices_agree}
If $s \in \cS$ then $\indexhigher_{ \cT }( s ) = \index_{ \cT }( s )$.
\end{Proposition}

\begin{proof}
For $s \in \cS$, consider the $( d+2 )$-angle \eqref{equ:index2_higher_angle} from Definition \ref{def:index2}.  By \cite[thm.\ 1]{GKO} it gives a tower of triangles in $\cC$,
\[
  \xymatrix @-1.3pc @! {
    & t_{ d-1 } \ar[rr] \ar[dr] & & t_{ d-2 } \ar[rr] \ar[dr] & & \cdots \ar[rr] & & t_1 \ar[rr] \ar[dr] & & t_0 \ar[dr] \\
    t_d \ar[ur] && v_{ d-1.5 } \ar[ur] \ar@{~>}[ll] & & v_{ d-2.5 } \ar@{~>}[ll] & \cdots & v_{ 1.5 } \ar[ur] & & v_{ 0.5 } \ar[ur] \ar@{~>}[ll] && s \lefteqn{,} \ar@{~>}[ll] \\
               }
\]
with $t_i \in \cT$.  Lemma \ref{lem:tower} below implies
\[
  \index_{ \cT }( s )
  = \sum_{ i=0 }^d (-1)^i \index_{ \cT }( t_i ),
\]
which by Lemma \ref{lem:index_additive0} reads
\[
  \index_{ \cT }( s )
  = \sum_{ i=0 }^d (-1)^i [t_i].
\]
Combining with Definition \ref{def:index2} proves the proposition.
\end{proof}

\begin{Lemma}
\label{lem:tower}
Consider a tower of triangles in $\cC$,
\[
  \xymatrix @-1.2pc @! {
    & c_d \ar[rr] \ar[dr] & & c_{ d-1 } \ar[rr] \ar[dr] & & \cdots \ar[rr] & & c_2 \ar[rr] \ar[dr] & & c_1 \ar[dr] \\
    c_{ d+1 } \ar[ur] && v_{ d-0.5 } \ar[ur] \ar@{~>}[ll] & & v_{ d-1.5 } \ar@{~>}[ll] & \cdots & v_{ 2.5 } \ar[ur] & & v_{ 1.5 } \ar[ur] \ar@{~>}[ll] && c_0 \lefteqn{.} \ar@{~>}[ll] \\
               }
\]
If $c_{ d+1 } \in \cT$ and $c_i \in \cS$ for $2 \leqslant i \leqslant d$, then
\[
  \sum_{ i=0 }^{ d+1 } (-1)^i \index_{ \cT }( c_i ) = 0.
\]
\end{Lemma}

\begin{proof}
Writing $v_{ d+0.5 } = c_{ d+1 }$ and $v_{ 0.5 } = c_0$, the triangles in the tower are 
\begin{equation}
\label{equ:tower1}
  v_{ i+0.5 } \rightarrow c_i \rightarrow v_{ i-0.5 } \xrightarrow{ \gamma_i } \Sigma v_{ i+0.5 }  \;\;\mbox{for}\;\;  1 \leqslant i \leqslant d.
\end{equation}
We claim that it is enough to prove
\begin{equation}
\label{equ:tower2}
  F( \Sigma v_{ i+0.5 } ) = 0  \;\;\mbox{for}\;\;  1 \leqslant i \leqslant d.
\end{equation}
Namely, Equation \eqref{equ:tower2} implies $F\gamma_i = 0$ for $1 \leqslant i \leqslant d$ whence Lemma \ref{lem:index_additive3} applied to the triangles \eqref{equ:tower1} gives
\[
  \index_{ \cT }( c_i ) = \index_{ \cT }( v_{ i+0.5 } ) + \index_{ \cT }( v_{ i-0.5 } )  \;\;\mbox{for}\;\;  1 \leqslant i \leqslant d.
\]
The alternating sum of these equations gives the equation in the lemma.

To prove Equation \eqref{equ:tower2}, note that
\begin{equation}
\label{equ:tower3}
  v_{ d+0.5 } = c_{ d+1 } \in \cT.
\end{equation}
Starting with this, descending induction using the triangles \eqref{equ:tower1} and $c_i \in \cS$ for $2 \leqslant i \leqslant d$
shows 
\begin{equation}
\label{equ:tower4}
  v_{ i+0.5 } \in \cS * \cdots * \Sigma^{ d-i-1 } \cS * \Sigma^{ d-i }\cT  \;\;\mbox{for} \;\; 1 \leqslant i \leqslant d-1.
\end{equation}

Equation \eqref{equ:tower3} implies
\begin{equation}
\label{equ:tower5}
  F( \Sigma v_{ d+0.5 } ) = 0
\end{equation}
by Lemma \ref{lem:vanishing}.  Equation \eqref{equ:tower4} implies
\begin{equation}
\label{equ:tower6}
  F( \Sigma v_{ i+0.5 } ) = 0  \;\;\mbox{for} \;\;  1 \leqslant i \leqslant d-1,
\end{equation}
because $F( - ) = \cC( t,- )$ vanishes on each of
\[
  \Sigma \cS, \ldots, \Sigma^{ d-i } \cS, \Sigma^{ d-i+1 }\cT \mbox{ for } 1 \leqslant i \leqslant d-1
\]
since $t \in \cT \subseteq \cS$ while $\cS \subseteq \cC$ is $d$-cluster tilting and $\cT \subseteq \cC$ is $2d$-cluster tilting.  Equations \eqref{equ:tower5} and \eqref{equ:tower6} prove Equation \eqref{equ:tower2}.
\end{proof}

\begin{Lemma}
\label{lem:Sigma}
If $s \in \cS$ and $0 \leqslant i \leqslant d-1$ is an integer, then $\index_{ \cT }( \Sigma^i s ) = (-1)^i \index_{ \cT }( s )$.
\end{Lemma}

\begin{proof}
For each integer $i$ there is a triangle $\Sigma^i s \rightarrow 0 \rightarrow \Sigma^{ i+1 }s \xrightarrow{ \id } \Sigma^{ i+1 }s$
in $\cC$.  If $0 \leqslant i \leqslant d-2$ then $F( \Sigma^{ i+1 }s ) = \cC( t,\Sigma^{ i+1 }s ) = 0$ because $t \in \cT \subseteq \cS$ while $\cS \subseteq \cC$ is $d$-cluster tilting.  In particular, $F( \id ) = 0$ so Lemma \ref{lem:index_additive3} implies $\index_{ \cT }( \Sigma^{ i+1 }s ) = - \index_{ \cT }( \Sigma^i s )$, which gives the lemma.
\end{proof}

The following result does not rely on Setups \ref{set:blanket1} and \ref{set:blanket2}.  It holds if $\cC$ is an arbitrary triangulated category.

\begin{Lemma}
\label{lem:shorter_tower}
Let $m \geqslant 1$ be an integer and consider a tower of triangles in $\cC$,
\begin{equation}
\label{equ:shorter_tower1}
\vcenter{
  \xymatrix @-1.6pc @! {
    & c_m \ar[rr] \ar[dr] & & c_{ m-1 } \ar[rr] \ar[dr] & & \cdots \ar[rr] & & c_2 \ar[rr] \ar[dr] & & c_1 \ar[dr] \\
    c_{ m+1 } \ar[ur] && v_{ m-0.5 } \ar[ur] \ar@{~>}^{\gamma_{ m-0.5 }}[ll] & & v_{ m-1.5 } \ar@{~>}^{\gamma_{ m-1.5 }}[ll] & \cdots & v_{ 2.5 } \ar[ur] & & v_{ 1.5 } \ar[ur] \ar@{~>}^{\gamma_{ 1.5 }}[ll] && c_0 \lefteqn{.} \ar@{~>}^{\gamma_{ 0.5 }}[ll] \\
                       }
        }
\end{equation}
Set
\[
  \gamma = \Sigma^{ m-1 }( \gamma_{ m-0.5 } ) \circ \Sigma^{ m-2 }( \gamma_{ m-1.5 } ) \circ \cdots \circ \Sigma ( \gamma_{ 1.5 } ) \circ \gamma_{ 0.5 }
\]
and consider a triangle in $\cC$,
\[
  \Sigma^{ m-1 }c_{ m+1 } \rightarrow y \rightarrow c_0 \xrightarrow{ \gamma } \Sigma^m c_{ m+1 }.
\]
If $m = 1$ then $y \cong c_1$.  If $m \geqslant 2$ then there is a tower of triangles in $\cC$,
\begin{equation}
\label{equ:shorter_tower2}
\vcenter{
  \xymatrix @-1.7pc @! {
    & c_{ m-1 } \ar[rr] \ar[dr] & & c_{ m-2 } \ar[rr] \ar[dr] & & \cdots \ar[rr] & & c_2 \ar[rr] \ar[dr] & & c_1 \ar[dr] \\
    c_m \ar[ur] && w_{ m-1.5 } \ar[ur] \ar@{~>}[ll] & & w_{ m-2.5 } \ar@{~>}[ll] & \cdots & w_{ 2.5 } \ar[ur] & & w_{ 1.5 } \ar[ur] \ar@{~>}[ll] && y \lefteqn{.} \ar@{~>}[ll] \\
                       }
        }
\end{equation}
\end{Lemma}

\begin{proof}
If $m = 1$ then the tower \eqref{equ:shorter_tower1} is a single triangle $c_2 \rightarrow c_1 \rightarrow c_0 \xrightarrow{ \gamma_{ 0.5 } } \Sigma c_2$, and there are no objects $v_{ \ast }$.  Since $\gamma = \gamma_{ 0.5 }$, part (i) follows.

If $m = 2$ then the tower \eqref{equ:shorter_tower1} is
\[
  \xymatrix @-0.5pc @! {
    & c_2 \ar[rr] \ar[dr] & & c_1 \ar[dr] \\
    c_3 \ar[ur] & & v_{ 1.5 } \ar[ur] \ar@{~>}^{\gamma_{ 1.5 }}[ll] && c_0 \lefteqn{.} \ar@{~>}^{\gamma_{ 0.5 }}[ll] \\
               }
\]
We have $\gamma = \Sigma( \gamma_{ 1.5 } ) \circ \gamma_{ 0.5 }$, so the octahedral axiom gives the following diagram where each row and column is a triangle.
\[
  \xymatrix @-1.1pc @! {
    c_2 \ar@{=}[r] \ar[d] & c_2 \ar[r] \ar[d] & 0 \ar[r] \ar[d] & \Sigma c_2 \ar[d] \\
    v_{ 1.5 } \ar[r] \ar[d] & c_1 \ar[r] \ar[d] & c_0 \ar[r]^-{ \gamma_{ 0.5 } } \ar@{=}[d] & \Sigma v_{ 1.5 } \ar[d]^{ \Sigma( \gamma_{ 1.5 } ) } \\
    \Sigma c_3 \ar[r] \ar[d] & y \ar[r] \ar[d] & c_0 \ar[r]_-{ \gamma } \ar[d] & \Sigma^2 c_3 \ar[d] \\
    \Sigma c_2 \ar@{=}[r] & \Sigma c_2 \ar[r] & 0 \ar[r] & \Sigma^2 c_2 \\
                      }
\]
The second column, $c_2 \rightarrow c_1 \rightarrow y \rightarrow \Sigma c_2$, gives the single triangle in the tower \eqref{equ:shorter_tower2}, and there are no objects $w_{ \ast }$.  

If $m \geqslant 3$ and the tower \eqref{equ:shorter_tower1} are given, then consider the truncated tower
\[
  \xymatrix @-1.5pc @! {
    & c_m \ar[rr] \ar[dr] & & c_{ m-1 } \ar[rr] \ar[dr] & & \cdots \ar[rr] & & c_3 \ar[rr] \ar[dr] & & c_2 \ar[dr] \\
    c_{ m+1 } \ar[ur] && v_{ m-0.5 } \ar[ur] \ar@{~>}^{\gamma_{ m-0.5 }}[ll] & & v_{ m-1.5 } \ar@{~>}^{\gamma_{ m-1.5 }}[ll] & \cdots & v_{ 3.5 } \ar[ur] & & v_{ 2.5 } \ar[ur] \ar@{~>}^{\gamma_{ 2.5 }}[ll] && v_{ 1.5 } \lefteqn{.} \ar@{~>}^{\gamma_{ 1.5 }}[ll] \\
               }
\]
Set
\[
  \gamma' = \Sigma^{ m-2 }( \gamma_{ m-0.5 } ) \circ \Sigma^{ m-3 }( \gamma_{ m-1.5 } ) \circ \cdots \circ \Sigma ( \gamma_{ 2.5 } ) \circ \gamma_{ 1.5 }
\]
and consider a triangle in $\cC$,
\[
  \Sigma^{ m-2 } c_{ m+1 } \rightarrow y' \rightarrow v_{ 1.5 } \xrightarrow{ \gamma' } \Sigma^{ m-1 } c_{ m+1 }.
\]
By induction, there is a tower
\begin{equation}
\label{equ:shorter_tower3}
\vcenter{
  \xymatrix @-1.7pc @! {
    & c_{ m-1 } \ar[rr] \ar[dr] & & c_{ m-2 } \ar[rr] \ar[dr] & & \cdots \ar[rr] & & c_3 \ar[rr] \ar[dr] & & c_2 \ar[dr] \\
    c_m \ar[ur] && w_{ m-1.5 } \ar[ur] \ar@{~>}[ll] & & w_{ m-2.5 } \ar@{~>}[ll] & \cdots & w_{ 3.5 } \ar[ur] & & w_{ 2.5 } \ar[ur] \ar@{~>}[ll] && y' \lefteqn{.} \ar@{~>}[ll] \\
                       }
        }
\end{equation}
We have $\gamma = \Sigma( \gamma' ) \circ \gamma_{ 0.5 }$, so the octahedral axiom gives the following diagram where each row and column is a triangle.
\[
  \xymatrix @-3.1pc @! {
    y' \ar@{=}[r] \ar[d] & y' \ar[r] \ar[d] & 0 \ar[r] \ar[d] & \Sigma y' \ar[d] \\
    v_{ 1.5 } \ar[r] \ar[d] & c_1 \ar[r] \ar[d] & c_0 \ar[r]^-{ \gamma_{ 0.5 } } \ar@{=}[d] & \Sigma v_{ 1.5 } \ar[d]^{ \Sigma( \gamma' ) } \\
    \Sigma^{ m-1 } c_{ m+1 } \ar[r] \ar[d] & y \ar[r] \ar[d] & c_0 \ar[r]_-{ \gamma } \ar[d] & \Sigma^m c_{ m+1 } \ar[d] \\
    \Sigma y' \ar@{=}[r] & \Sigma y' \ar[r] & 0 \ar[r] & \Sigma^2 y' \\
                      }
\]
The second column, $y' \rightarrow c_1 \rightarrow y \rightarrow \Sigma y'$, can be concatenated with the tower \eqref{equ:shorter_tower3}, giving the tower \eqref{equ:shorter_tower2} with $w_{ 1.5 } = y'$.
\end{proof}

The following is Theorem \ref{thm:index_additive5_intro} from the introduction.

\begin{Theorem}
\label{thm:index_additive5}
\begin{enumerate}
\setlength\itemsep{4pt}

  \item  If
\begin{equation}
\label{equ:index_additive5_1}
  s_{ d+1 } \rightarrow \cdots \rightarrow s_0 \xrightarrow{ \gamma } \Sigma^d s_{ d+1 }
\end{equation}
is a $( d+2 )$-angle in $\cS$, then
\[
  \sum_{ i=0 }^{ d+1 } (-1)^i \indexhigher_{ \cT }( s_i ) = 
  \theta \big( [ \Image F\gamma ] \big).
\]

  \item  The property in part (i) determines the homomorphism $\theta : \K0( \mod\,\Gamma ) \rightarrow \Ksp( \cT )$ of abelian groups.

\end{enumerate}

\end{Theorem}

\begin{proof}
(i):  If $d = 1$ then $\cS = \cC$ and Proposition \ref{pro:indices_agree} says $\indexhigher_{ \cT } = \index_{ \cT }$, so part (i) of the theorem reduces to Theorem \ref{thm:index_additive4}.

Assume $d \geqslant 2$.  By \cite[thm.\ 1]{GKO}, the $( d+2 )$-angle \eqref{equ:index_additive5_1} gives a tower of triangles in $\cC$,
\[
  \xymatrix @-1.4pc @! {
    & s_d \ar[rr] \ar[dr] & & s_{ d-1 } \ar[rr] \ar[dr] & & \cdots \ar[rr] & & s_2 \ar[rr] \ar[dr] & & s_1 \ar[dr] \\
    s_{ d+1 } \ar[ur] && v_{ d-0.5 } \ar[ur] \ar@{~>}^{\gamma_{ d-0.5 }}[ll] & & v_{ d-1.5 } \ar@{~>}^{\gamma_{ d-1.5 }}[ll] & \cdots & v_{ 2.5 } \ar[ur] & & v_{ 1.5 } \ar[ur] \ar@{~>}^{\gamma_{ 1.5 }}[ll] && s_0 \lefteqn{,} \ar@{~>}^{\gamma_{ 0.5 }}[ll] \\
               }
\]
with 
\[
  \gamma = \Sigma^{ d-1 }( \gamma_{ d-0.5 } ) \circ \Sigma^{ d-2 }( \gamma_{ d-1.5 } ) \circ \cdots \circ \Sigma( \gamma_{ 1.5 } ) \circ \gamma_{ 0.5 }.
\]
By Lemma \ref{lem:shorter_tower} there is a triangle in $\cC$,
\begin{equation}
\label{equ:index_additive5_2}
  \Sigma^{ d-1 }s_{ d+1 } \rightarrow y \rightarrow s_0 \xrightarrow{ \gamma } \Sigma^d s_{ d+1 },
\end{equation}
and a tower of triangles in $\cC$,
\begin{equation}
\label{equ:index_additive5_3}
\vcenter{
  \xymatrix @-1.6pc @! {
    & s_{ d-1 } \ar[rr] \ar[dr] & & s_{ d-2 } \ar[rr] \ar[dr] & & \cdots \ar[rr] & & s_2 \ar[rr] \ar[dr] & & s_1 \ar[dr] \\
    s_d \ar[ur] && w_{ d-1.5 } \ar[ur] \ar@{~>}[ll] & & w_{ d-2.5 } \ar@{~>}[ll] & \cdots & w_{ 2.5 } \ar[ur] & & w_{ 1.5 } \ar[ur] \ar@{~>}[ll] && y \lefteqn{.} \ar@{~>}[ll] \\
                       }
        }
\end{equation}

By Theorem \ref{thm:index_additive4}, the triangle \eqref{equ:index_additive5_2} gives
\[
  \index_{ \cT }( y )
  = \index_{ \cT }( \Sigma^{ d-1 }s_{ d+1 } )
  + \index_{ \cT }( s_0 )
  - \theta \big( [ \Image F\gamma ] \big),
\]
which by Lemma \ref{lem:Sigma} reads
\[
  \index_{ \cT }( y )
  = (-1)^{ d-1 }\index_{ \cT }( s_{ d+1 } )
  + \index_{ \cT }( s_0 )
  - \theta \big( [ \Image F\gamma ] \big).
\]
By Lemma \ref{lem:tower}, the tower \eqref{equ:index_additive5_3} gives
\[
  \index_{ \cT }( y ) + \sum_{ i=1 }^d (-1)^i \index_{ \cT }( s_i ) = 0.
\]
Note that the condition $c_{ d+1 } \in \cT$ in Lemma \ref{lem:tower} is satisfied because we have to set $c_{ d+1 } = 0$.  Combining the last two displayed equations gives
\[
  \sum_{ i=0 }^{ d+1 } (-1)^i \index_{ \cT }( s_i ) = 
  \theta \big( [ \Image F\gamma ] \big).
\]
Applying Proposition \ref{pro:indices_agree} completes the proof of part (i).

(ii):  Recall the notion of dualizing $k$-variety, see \cite[p.\ 307]{AR2}.  By \cite[p.\ 307]{AR2} the category $\cT = \add( t )$ is a dualizing $k$-variety because it is equivalent to the category of finitely generated right modules over $\Gamma = \End_{ \cC }( t )$.  Hence $\cT$ has right and left almost split morphisms by \cite[prop.\ 2.10(2)]{IY}.

Let $s_0$ be an indecomposable summand of $t$.  A right almost split morphism $s_1 \rightarrow s_0$   in $\cT$ is part of a $( d+2 )$-angle \eqref{equ:index_additive5_1} in $\cS$ by \cite[def.\ 2.1(F1)(c) and (F2)]{GKO}.  The functor $F$ sends \eqref{equ:index_additive5_1} to a long exact sequence by \cite[rmks.\ 2.2(c) and prop.\ 2.5(a)]{GKO}, so in particular $\Image F\gamma \cong S( s_0 )$ is the simple $\Gamma$-right module corresponding to $s_0$.  Part (i) hence says
\[
  \theta \big( [ S( s_0 ) ] \big)
  = \sum_{ i=0 }^{ d+1 } (-1)^i \indexhigher_{ \cT }( s_i ).
\]
This implies part (ii) because the classes $[ S( s_0 ) ]$ generate the abelian group $\K0( \mod\,\Gamma )$, since each simple $\Gamma$-right module has the form $S( s_0 )$ for some indecomposable summand $s_0$ of $t$.
\end{proof}

\section{Tropical friezes on $( d+2 )$-angulated categories}
\label{sec:friezes}

This section proves Theorem \ref{thm:frieze_intro} from the introduction (= Theorem \ref{thm:frieze}).  Recall that $\cS$ is said to be $2d$-Calabi--Yau if there are natural isomorphisms $\cS( s,s' ) \cong \dual\!\cS\big( s',( \Sigma^d )^2 s \big)$ for $s,s' \in \cS$.

\begin{Lemma}
\label{lem:dichotomy}
Assume that $\cS$ is $2d$-Calabi--Yau.  Let $s_0, s_{ d+1 } \in \indec\, \cS$ be an exchange pair with ensuing $( d+2 )$-angles \eqref{equ:exchange_pair1} and \eqref{equ:exchange_pair2}; see Section \ref{subsec:frieze}.  Then $F\gamma_0 = 0$ or $F\gamma_{ d+1 } = 0$.   
\end{Lemma}

\begin{proof}
Suppose $F\gamma_0 \neq 0$, that is $\cC( t,\gamma_0 ) \neq 0$.  Then we can pick a morphism $t \xrightarrow{ \tau } s_0$ such that the composition $t \xrightarrow{ \tau } s_0 \xrightarrow{ \gamma_0 } \Sigma^d s_{ d+1 }$ is non-zero.  It follows that the induced homomorphism
\[
  \cS( \Sigma^d s_0,\Sigma^{ 2d }s_{ d+1 } )
  \xrightarrow{ ( \Sigma^d \tau )^* }
  \cS( \Sigma^d t,\Sigma^{ 2d }s_{ d+1 } )
\]
is non-zero.  Hence the lower horizontal homomorphism is non-zero in the commutative square
\[
  \xymatrix @R=-3.1pc @C=-2.4pc @! {
    \cS( s_{ d+1 },\Sigma^d t ) \ar^{ (\Sigma^d \tau)_* }[rr] \ar^{ \mbox{\rotatebox{90}{$\sim$}} }[d] && \cS( s_{ d+1 },\Sigma^d s_0 ) \ar^{ \mbox{\rotatebox{90}{$\sim$}} }[d] \\
    \dual\!\cS( \Sigma^d t,\Sigma^{ 2d }s_{ d+1 } ) \ar_{ \dual\big( ( \Sigma^d \tau )^* \big) }[rr] && \dual\!\cS ( \Sigma^d s_0,\Sigma^{ 2d }s_{ d+1 } )\lefteqn{,} \\
                       }
\]
which exists since $\cS$ is $2d$-Calabi--Yau.  This implies that the upper horizontal homomorphism is non-zero, hence surjective since the target has dimension $1$ over $k$ by assumption.  So there exists a morphism $s_{ d+1 } \xrightarrow{ \rho } \Sigma^d t$ such that $\gamma_{ d+1 } = ( \Sigma^d \tau )_*( \rho )$, that is, $s_{ d+1 } \xrightarrow{ \gamma_{ d+1 } } \Sigma^d s_0$ has been factored as $s_{ d+1 } \xrightarrow{ \rho } \Sigma^d t \xrightarrow{ \Sigma^d \tau } \Sigma^d s_0$.  Then each composition $t \rightarrow s_{ d+1 } \xrightarrow{ \gamma_{ d+1 } } \Sigma^d s_0$ is zero because $\cC( t,\Sigma^d t ) = 0$, so $\cC( t,\gamma_{ d+1 } ) = 0$, that is, $F\gamma_{ d+1 } = 0$.
\end{proof}

The following is Theorem \ref{thm:frieze_intro} from the introduction.

\begin{Theorem}
\label{thm:frieze}
Assume that $\cS$ is $2d$-Calabi--Yau and that $d$ is odd, and let $\varphi : \Ksp( \cT ) \rightarrow \BZ$ be a homomorphism of abelian groups.  The composition 
\[
  \varphi \circ \indexhigher_{ \cT } : \obj\, \cS \rightarrow \BZ
\]
is a tropical frieze if $\varphi$ satisfies
\begin{equation}
\label{equ:frieze0}
  \varphi\theta\big( [M] \big) \geqslant 0
  \;\; \mbox{for each} \;\;
  M \in \mod( \Gamma ).
\end{equation}
\end{Theorem}

\begin{proof}
It is clear that $f = \varphi \circ \indexhigher_{ \cT }$ is constant on isomorphism classes and satisfies $f( s \oplus s' ) = f( s ) + f( s' )$.  Let $s_0, s_{ d+1 } \in \indec\, \cS$ be an exchange pair with ensuing $( d+2 )$-angles \eqref{equ:exchange_pair1} and \eqref{equ:exchange_pair2}, and set
\[
  S = f( s_0 ) + (-1)^{ d+1 }f( s_{ d+1 } ) \;,\;
  X = \sum_{ i=1 }^d (-1)^{ i+1 }f( x_i ) \;,\;
  Y = \sum_{ i=1 }^d (-1)^{ i+1 }f( y_i )
\]
whence the theorem amounts to
\begin{equation}
\label{equ:frieze1}	
  S = \max \{ X,Y \}.
\end{equation}

Set
\[
  C_0 = \varphi\theta \big( [ \Image F\gamma_0 ] \big) \;\;,\;\;
  C_{ d+1 } = \varphi\theta \big( [ \Image F\gamma_{ d+1 } ] \big)
\]
and observe that
\begin{eqnarray}
\label{equ:frieze2}
  & \mbox{$C_0, C_{ d+1 } \geqslant 0$ by Equation \eqref{equ:frieze0},} & \\[1mm]
\label{equ:frieze3}
  & \mbox{$C_0 = 0$ or $C_{ d+1 } = 0$ by Lemma \ref{lem:dichotomy}.} &
\end{eqnarray}

We have
\begin{align*}
  & f( s_0 )
  + \sum_{ i=1 }^d (-1)^i f( x_i )
  + (-1)^{ d+1 }f( s_{ d+1 } ) \\
  & \;\; =
  \varphi \Bigg(
  \indexhigher_{ \cT }( s_0 )
  + \sum_{ i=1 }^d (-1)^i \indexhigher_{ \cT }( x_i )
  + (-1)^{ d+1 }\indexhigher_{ \cT }( s_{ d+1 } )
  \Bigg) \\
  & \;\; \stackrel{ \rm (a) }{=}
  \varphi\theta \big( [ \Image F\gamma_0 ] \big),
\end{align*}
where (a) is by Theorem \ref{thm:index_additive5} applied to the $( d+2 )$-angle \eqref{equ:exchange_pair1}.  In the above notation, this means
\[
  S = X + C_0.
\]
A similar computation with the $( d+2 )$-angle \eqref{equ:exchange_pair2} shows
\[
  S = Y - (-1)^d C_{ d+1 }.
\]

Since $d$ is odd, Equation \eqref{equ:frieze1} is now an elementary consequence of the last two displayed equations combined with Equations \eqref{equ:frieze2} and \eqref{equ:frieze3}.
\end{proof}

\section{An example of a tropical frieze on a $5$-angulated category}
\label{sec:example}

This section shows an example of a $5$-angulated category $\cS$ and some tropical friezes on $\cS$ arising from Theorem \ref{thm:frieze_intro}.  One of the friezes is shown in Figure \ref{fig:frieze2} in the introduction.

Let $Q$ be the quiver $4 \rightarrow 3 \rightarrow 2 \rightarrow 1$, and let $\rad$ be the radical of the path algebra $kQ$.  The algebra $\Phi = kQ / \rad^2$ is $3$-representation finite in the sense of \cite[def.\ 2.2]{IO1}, and the unique $3$-cluster tilting subcategory of $\mod( \Phi )$ is
\[
  \cF = \add( \Phi \oplus \dual\!\Phi ),
\]
see \cite[prop.\ 2.3]{IO1} and \cite[thm.\ 3]{V}.  There is an associated $5$-angulated higher dimensional cluster category $\cS$ with AR quiver shown in Figure \ref{fig:AR_quiver}, see \cite[sec.\ 5]{OT} or Section \ref{subsec:classic2}.
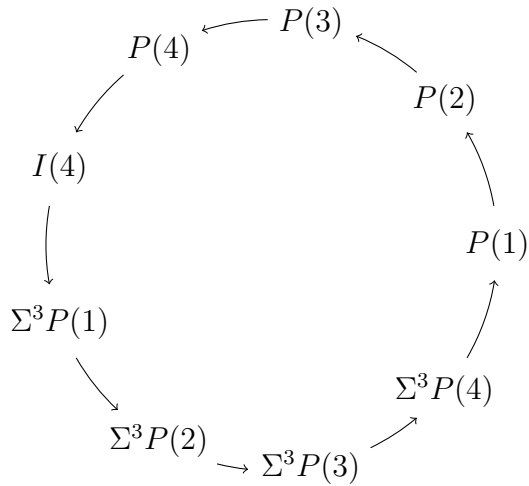
\begin{figure}
\begin{tikzpicture}[scale=3]
  \node at (0:1.0){$P(1)$};
  \draw[->] (10:1.0) arc (10:30:1.0);
  \node at (40:1.0){$P(2)$};
  \draw[->] (50:1.0) arc (50:68:1.0);
  \node at (80:1.0){$P(3)$};
  \draw[->] (91:1.0) arc (91:108:1.0);
  \node at (120:1.0){$P(4)$};
  \draw[->] (131:1.0) arc (131:150:1.0);
  \node at (160:1.0){$I(4)$};
  \draw[->] (170:1.0) arc (170:190:1.0);
  \node at (200:1.0){$\Sigma^3 P(1)$};
  \draw[->] (210:1.0) arc (210:227:1.0);
  \node at (240:1.0){$\Sigma^3 P(2)$};
  \draw[->] (256:1.0) arc (256:264:1.0);
  \node at (280:1.0){$\Sigma^3 P(3)$};
  \draw[->] (296:1.0) arc (296:310:1.0);
  \node at (320:1.0){$\Sigma^3 P(4)$};
  \draw[->] (330:1.0) arc (330:351:1.0);
\end{tikzpicture}
\caption{The AR quiver of the $5$-angulated category $\cS$ from Section \ref{sec:example}.}
\label{fig:AR_quiver}
\end{figure}
There is an inclusion
\[
  \cF \subseteq \cS
\]  
by \cite[thm.\ 5.2]{OT}, and
\[  
  t = \Phi
\]  
is an Oppermann--Thomas cluster tilting object of $\cS$ by \cite[thm.\ 5.5]{OT}.  We set $\cT = \add( t )$.

The theory of this paper applies to $\cS$ and $\cT$ with $d = 3$, because there is a triangulated category $\cC$ such that Setups \ref{set:blanket1} and \ref{set:blanket2} are satisfied, see \cite[thm.\ 5.25]{OT}.  Indeed, $\cC$ is the $6$-Calabi--Yau Amiot cluster category defined as  the triangulated hull of the orbit category $\Db( \mod\,\Phi )/( S\Sigma^{ -6 } )$ where $S$ is the Serre functor, see \cite[sec.\ 4.1]{A}, \cite[sec.\ 2.1]{IO2}, and \cite[sec.\ 5.1.2]{OT}.

We recall some properties of $\Phi$, $\cS$, and $t$.

\subsection{The algebra $\Phi$}

For $1 \leqslant i \leqslant 4$, the simple $\Phi$-right module at vertex $i$ will be  denoted $S( i )$.  Its projective cover and injective envelope are $P( i ) = e_i\Phi$ and $I( i ) = \dual( \Phi e_i )$, where $e_i$ is the idempotent at $i$.  The corresponding representations of $Q$ are the following. 
\[
\mbox{
\begin{tabular}{c|c|c}
  Indecomposable projective & Representation & Indecomposable injective \\ \cline{1-3}
  $P(1)$ & $0 \rightarrow 0 \rightarrow 0 \rightarrow k$ &        \\[1mm]
  $P(2)$ & $0 \rightarrow 0 \rightarrow k \rightarrow k$ & $I(1)$ \\[1mm]
  $P(3)$ & $0 \rightarrow k \rightarrow k \rightarrow 0$ & $I(2)$ \\[1mm]
  $P(4)$ & $k \rightarrow k \rightarrow 0 \rightarrow 0$ & $I(3)$ \\[1mm]
         & $k \rightarrow 0 \rightarrow 0 \rightarrow 0$ & $I(4)$ 
\end{tabular}
}
\]
The only indecomposable injective right module which is not projective is $I(4)$, so 
\[
  \indec\,\cF = \{ P( 1 ), P( 2 ), P( 3 ), P( 4 ), I( 4 ) \}.
\]

\subsection{The $5$-angulated category $\cS$}
\label{subsec:OT1}

The category $\cS$ is $6$-Calabi--Yau by \cite[thm.\ 5.2]{OT}.

Let $\cP = \add( \Phi )$ denote the projective modules in $\mod( \Phi )$.  The indecomposable objects of $\cS$ are $\indec\, \cS = \indec\, \cF \cup \indec\, \Sigma^3\cP$ by \cite[thm.\ 5.2(1)]{OT}, that is
\[
  \indec\, \cS = \{ P( 1 ), P( 2 ), P( 3 ), P( 4 ), I( 4 ),
                    \Sigma^3 P( 1 ), \Sigma^3 P( 2 ), \Sigma^3 P( 3 ), \Sigma^3 P( 4 )\}.
\]

The morphisms in $\cS$ are described in \cite[thm.\ 5.3(3)]{OT}.  This permits us to compute the AR quiver of $\cS$, which is shown in Figure \ref{fig:AR_quiver}.  We get that for $s,s' \in \indec\,\cS$,
\begin{equation}
\label{equ:Hom}
  \cS( s,s' ) =
  \left\{
    \begin{array}{cl}
      k & \mbox{if there is a arrow $s \rightarrow s'$ in the AR quiver,} \\[2mm]
      0 & \mbox{otherwise.}
    \end{array}
  \right.
\end{equation}
The action of $\Sigma^3$ on indecomposable objects is given by moving $4$ steps clockwise in the quiver.

Each $3$-extension in $\cF$ induces a $5$-angle in $\cS$ by \cite[exa.\ 5.17]{OT}.  In particular, the non-trivial $3$-extension $0 \rightarrow P(1) \rightarrow P(2) \rightarrow P(3) \rightarrow P(4) \rightarrow I(4) \rightarrow 0$ in $\cF$ induces a $5$-angle 
\begin{equation}
\label{equ:basic_5-angle}
  P(1) \rightarrow P(2) \rightarrow P(3) \rightarrow P(4) \rightarrow I(4) \rightarrow \Sigma^3 P(1)
\end{equation}
in $\cS$ where $I(4) \rightarrow \Sigma^3 P(1)$ is non-zero.  By ``rotation'' of this, in the sense of \cite[def.\ 2.1(F2)]{GKO}, we get that if $s_4, \ldots, s_0$ are consecutive objects in the AR quiver of $\cS$, then there is a $5$-angle in $\cS$,
\begin{equation}
\label{equ:fancy_5-angle1}
  s_4 \rightarrow s_3 \rightarrow s_2 \rightarrow s_1 \rightarrow s_0 \xrightarrow{\gamma} \Sigma^3 s_4,
\end{equation}
with $\gamma \neq 0$.  Note that in particular, $s_0$ and $s_4$ are an exchange pair.

\subsection{The Oppermann--Thomas cluster tilting object $t$}
\label{subsec:CT}

The object $t = \Phi$ of $\cS$ is in $\cF$, so the endomorphism algebra is
\[
  \Gamma = \End_{ \cS }( t ) \cong \End_{ \cF }( \Phi ) \cong \Phi.
\]  
Hence $F( - ) = \cS( t,- )$ is a functor
\[
  F : \cS \rightarrow \mod( \Phi ).
\]

We proceed to apply the theory of Sections \ref{sec:higher} and \ref{sec:friezes} to $\cS$ and $\cT$.

\subsection{The $( d+2 )$-angulated index with respect to $t$}
\label{subsec:indices}

The $5$-angle \eqref{equ:basic_5-angle} shows $\indexhigher_{ \cT }\big( I(4) \big) = - P(1) + P(2) - P(3) + P(4)$ (we omit square brackets in elements of $\Ksp( \cT )$ for readability).  The indices of the remaining indecomposable objects of $\cS$ can be computed using Lemma \ref{lem:index_additive0} and Proposition \ref{pro:indices_agree}, and we get the following.
\[
\mbox{
\begin{tabular}{c|c}
  $s$               & $\indexhigher_{ \cT }( s )$ \\ \cline{1-2}
  $P( 1 )$          & $P( 1 )$   \\[1mm]
  $P( 2 )$          & $P( 2 )$   \\[1mm]
  $P( 3 )$          & $P( 3 )$   \\[1mm]
  $P( 4 )$          & $P( 4 )$   \\[1mm]
  $I( 4 )$          & $- P( 1 ) + P( 2 ) - P( 3 ) + P( 4 )$ \\[1mm]
  $\Sigma^3 P( 1 )$ & $- P( 1 )$ \\[1mm]
  $\Sigma^3 P( 2 )$ & $- P( 2 )$ \\[1mm]
  $\Sigma^3 P( 3 )$ & $- P( 3 )$ \\[1mm]
  $\Sigma^3 P( 4 )$ & $- P( 4 )$ \\
\end{tabular}
}
\]

\subsection{The map $\theta$}
\label{subsec:theta}

The following table lists some of the $5$-angles \eqref{equ:fancy_5-angle1}.  In each case, the last object is indeed $\Sigma^3$ of the first, as one can check using that the action of $\Sigma^3$ is given by moving $4$ steps clockwise in the AR quiver of $\cS$.  For each $5$-angle we can compute $\Image F\gamma$ using Equation \eqref{equ:Hom}.  We can also compute $\theta \big( [ \Image F\gamma ] \big)$ by using that it is equal to $\sum_{ i=0 }^4 (-1)^i \indexhigher_{ \cT }( s_i )$ by Theorem \ref{thm:index_additive5};  this sum can be evaluated using Section \ref{subsec:indices}.  
\[
\mbox{
\begin{tabular}{c|c|c}
  $s_4 \rightarrow s_3 \rightarrow s_2 \rightarrow s_1 \rightarrow s_0 \xrightarrow{ \gamma } \Sigma^3 s_4$ & $\Image F\gamma$ & $\theta \big( [ \Image F\gamma ] \big)$ \\ \cline{1-3}
  $\Sigma^3 P( 1 ) \rightarrow \Sigma^3 P( 2 ) \rightarrow \Sigma^3 P( 3 ) \rightarrow \Sigma^3 P( 4 ) \rightarrow P( 1 ) \xrightarrow{ \gamma } P( 2 )$ & $S( 1 )$ & $P( 2 ) - P( 3 ) + P( 4 )$ \\
  $\Sigma^3 P( 2 ) \rightarrow \Sigma^3 P( 3 ) \rightarrow \Sigma^3 P( 4 ) \rightarrow P( 1 ) \rightarrow P( 2 ) \xrightarrow{ \gamma } P(3)$ & $S( 2 )$ & $- P( 1 ) + P( 3 ) - P( 4 )$ \\
  $\Sigma^3 P( 3 ) \rightarrow \Sigma^3 P( 4 ) \rightarrow P( 1 ) \rightarrow P( 2 ) \rightarrow P(3) \xrightarrow{ \gamma } P( 4 )$ & $S( 3 )$ & $P( 1 ) - P( 2 ) + P( 4 )$ \\
  $\Sigma^3 P( 4 ) \rightarrow P( 1 ) \rightarrow P( 2 ) \rightarrow P(3) \rightarrow P( 4 ) \xrightarrow{ \gamma } I( 4 )$ & $S( 4 )$ & $- P( 1 ) + P( 2 ) - P( 3 )$ \\
\end{tabular}
}
\]
This table determines the map $\theta$.

\subsection{Tropical friezes on $\cS$}
\label{subsec:frieze1}

Let $f : \obj\,\cS \rightarrow \BZ$ be a tropical frieze.  As mentioned in Section \ref{subsec:frieze}, this means that if $s_4, \ldots, s_0$ are consecutive objects in the AR quiver of $\cS$ (denoted $a,b,c,d,e$ in Section \ref{subsec:frieze}), then
\begin{equation}
\label{equ:frieze_specific}
  f( s_0 ) + f( s_4 )
  = \max \{\, f( s_1 ) - f( s_2 ) + f( s_3 ),0 \,\}.
\end{equation}
To show this, note that by Section \ref{subsec:OT1} the objects $s_0$ and $s_4$ are an exchange pair in $\cS$, and there is a $5$-angle \eqref{equ:fancy_5-angle1} in $\cS$.  Moreover, $s_4, \ldots, s_0$ occur in an anticlockwise order in the AR quiver of $\cS$, so there are four clockwise steps from $s_0$ to $s_4$ whence $\Sigma^3 s_0 \cong s_4$.  This gives a trivial $5$-angle in $\cS$,
\[
  s_0 \rightarrow 0 \rightarrow 0 \rightarrow 0 \rightarrow s_4 \xrightarrow{\cong} \Sigma^3 s_0.
\]
Equation \eqref{equ:frieze_specific} now follows from Definition \ref{def:frieze}.

\subsection{The tropical frieze $\varphi \circ \indexhigher_{ \cT }$}
\label{subsec:frieze2}

Let $\varphi : \Ksp( \cT ) \rightarrow \BZ$ be a homomorphism of abelian groups defined as follows.
\[
\mbox{
\begin{tabular}{c|c}
  $x$      & $\varphi( x )$ \\ \cline{1-2}
  $P( 1 )$ & $\alpha$ \\[1mm]
  $P( 2 )$ & $\beta$  \\[1mm]
  $P( 3 )$ & $\gamma$ \\[1mm]
  $P( 4 )$ & $\delta$ \\[1mm]
\end{tabular}
}
\]
The composition $\varphi \circ \indexhigher_{ \cT } : \obj\, \cS \rightarrow \BZ$ can be computed using the table in Section \ref{subsec:indices}.  Its values on the AR quiver of $\cS$ are shown in Figure \ref{fig:frieze}.
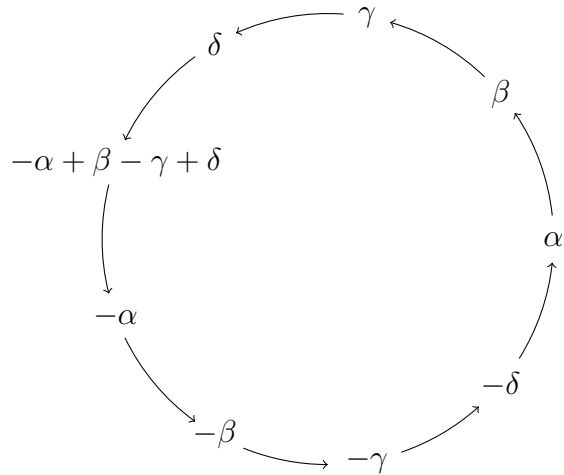
\begin{figure}
\begin{tikzpicture}[scale=3]
  \node at (0:1.0){$\alpha$};
  \draw[->] (6:1.0) arc (6:34:1.0);
  \node at (40:1.0){$\beta$};
  \draw[->] (46:1.0) arc (46:74:1.0);
  \node at (80:1.0){$\gamma$};
  \draw[->] (86:1.0) arc (86:114:1.0);
  \node at (120:1.0){$\delta$};
  \draw[->] (126:1.0) arc (126:154:1.0);
  \node at (160:1.0){$-\alpha+\beta-\gamma+\delta$};
  \draw[->] (166:1.0) arc (166:194:1.0);
  \node at (200:1.0){$-\alpha$};
  \draw[->] (206:1.0) arc (206:234:1.0);
  \node at (240:1.0){$-\beta$};
  \draw[->] (248:1.0) arc (248:270:1.0);
  \node at (280:1.0){$-\gamma$};
  \draw[->] (289:1.0) arc (289:312:1.0);
  \node at (320:1.0){$-\delta$};
  \draw[->] (328:1.0) arc (328:354:1.0);
\end{tikzpicture}
\caption{The AR quiver of $\cS$ with the values of the tropical frieze $\varphi \circ \indexhigher_{ \cT }$.  The integers $\alpha,\beta,\gamma,\delta$ must satisfy the inequalities \eqref{equ:inequalities}.}
\label{fig:frieze}
\end{figure}
It is a tropical frieze by Theorem \ref{thm:frieze} if
\begin{equation}
\label{equ:frieze4}
  \varphi\theta\big( [M] \big) \geqslant 0
  \;\; \mbox{for each} \;\;
  M \in \mod( \Gamma ).
\end{equation}
This is equivalent to $\varphi$ being non-negative on each of the elements in the last column of the table in Section \ref{subsec:theta}, which is again equivalent to
\begin{equation}
\label{equ:inequalities}	
  \begin{array}{rcc}
    \beta - \gamma + \delta    & \geqslant & 0, \\[1mm]
    - \alpha + \gamma - \delta & \geqslant & 0, \\[1mm]
    \alpha - \beta + \delta    & \geqslant & 0, \\[1mm]
    - \alpha + \beta - \gamma  & \geqslant & 0.
  \end{array}
\end{equation}
This system has solutions, for instance $( \alpha,\beta,\gamma,\delta ) = ( -17,-8,2,19 )$, which gives the tropical frieze shown in Figure \ref{fig:frieze2} in the introduction.

\medskip
\noindent
{\bf Acknowledgement.}
We thank Yann Palu for comments on a preliminary version, Quanshui Wu and Milen Yakimov for the invitation to present these results at the Joint International Meeting of the CMS and AMS in Shanghai, June 2018, and Zongzhu Lin for comments to that talk.  This work was supported by EPSRC grant EP/P016014/1 ``Higher Dimensional Homological Algebra''.

\end{document}